\documentclass[12pt]{amsproc}

\usepackage[margin=1in]{geometry}
\usepackage[utf8]{inputenc}
\usepackage[dvipsnames]{xcolor}
\usepackage[shortlabels]{enumitem}
\usepackage{graphicx}
\usepackage{mathrsfs}
\usepackage{amsmath,amssymb,amsthm,mathtools}
\usepackage{MnSymbol}
\usepackage{verbatim}
\usepackage{appendix}
\usepackage{units}
\usepackage{standalone}
\usepackage{dsfont}
\usepackage[normalem]{ulem}
\usepackage[hidelinks]{hyperref}
\usepackage{algorithm}
\usepackage[noend]{algpseudocode}
\usepackage{caption}
\usepackage{etoolbox}
\usepackage{booktabs}
\usepackage[hang,flushmargin]{footmisc}

\AtBeginEnvironment{algorithmic}{\setlength{\itemsep}{0pt}}

\usepackage{etoolbox} 
\usepackage{xparse}  

\algsetblockdefx[Section]{Section}{EndSection}{}{0pt}[2]%
{\hspace{-50cm} \hfill \rule[2.2ex]{#2}{.1pt}\vspace{-2ex}\\%
	\hfill\textsc{#1} \vspace{-3ex}} 
[0]{\vspace{-2.3ex}}

\algrenewcommand\algorithmicrequire{\textbf{Input:}}
\algrenewcommand\algorithmicensure{\textbf{Output:}}

\usepackage{hyperref} 
\usepackage{soul}
\usepackage{nicematrix}
\usepackage{pgfplots}
\pgfplotsset{compat=1.18}
\usetikzlibrary{pgfplots.groupplots}

\mathtoolsset{showonlyrefs}
\newtheorem{theorem}{Theorem}[section]
\newtheorem{corollary}[theorem]{Corollary} 
\newtheorem{lemma}[theorem]{Lemma}
\newtheorem{proposition}[theorem]{Proposition}

\newtheorem{remark}[theorem]{Remark}

\theoremstyle{definition}
\newtheorem{example}[theorem]{Example}
\newtheorem{definition}[theorem]{Definition}

\newtheorem*{theorem**}{Theorem\theoremnum}
\newenvironment{theorem*}[1][]{%
  \edef\theoremnum{\if\relax\detokenize{#1}\relax\else~#1\fi}
  \begin{theorem**}
}{%
  \end{theorem**}
}

\DeclareMathAlphabet\mathbscr{U}{eus}{b}{n}

\newcommand{\fa}{\mathbf{a}}
\newcommand{\fb}{\mathbf{b}}
\newcommand{\fc}{\mathbf{c}}
\newcommand{\fx}{\mathbf{x}}
\newcommand{\fv}{\mathbf{v}}
\newcommand{\fw}{\mathbf{w}}

\newcommand{\Acal}{\mathcal{A}}

\newcommand{\Ocal}{\mathcal{O}}
\newcommand{\CC}{\mathbb{C}}

\newcommand{\PP}{\mathbb{P}}
\newcommand{\RR}{\mathbb{R}}

\newcommand{\bT}{\mathbf{T}}

\DeclareMathOperator{\Span}{span}
\DeclareMathOperator{\rank}{rank}

\DeclareMathOperator{\colspan}{colspan}

\DeclareMathOperator{\Vect}{vect}

\DeclareMathOperator{\Diag}{Diag}

\newcommand{\T}{^\mathsf{T}}

\newcommand{\fe}{\mathbf{e}}
\newcommand{\SB}{\mathbb{S}}
\newcommand{\fu}{\mathbf{u}}
\newcommand{\fy}{\mathbf{y}}
\newcommand{\Bcal}{\mathcal{B}}
\newcommand{\bM}{\mathbf{M}}
\newcommand{\sT}{\mathbscr{T}}
\newcommand{\Id}{\mathbf{I}}
\newcommand{\sR}{\mathbscr{R}}
\newcommand{\sS}{\mathbscr{S}}

\providecommand{\keywords}[1]
{
  \small	
  \textit{Keywords: } #1
}

\title[Multi-subspace power method]{Multi-subspace power method
\\
for decomposing partially symmetric tensors}

\author[Kexin Wang]{Kexin Wang$^{1}$}
\author[Jo\~ao M. Pereira]{Jo\~ao M. Pereira$^{2}$}
\author[Joe Kileel]{Joe Kileel$^{3}$}
\author[Anna Seigal]{Anna Seigal$^{1,*}$}
\thanks{%
\noindent$^{1}$\;John A. Paulson School of Engineering and Applied Sciences, Harvard University, Cambridge, MA, USA\\
\noindent$^{2}$\;Department of Mathematics, University of Georgia, Athens, GA, USA
\\
\noindent$^{3}$\;Department of Mathematics, University of Texas at Austin, Austin, TX, USA
\\ \noindent$^{*}$\;Corresponding author: aseigal@seas.harvard.edu
}

\date{}


\begin{document}

\begin{abstract}
We present an algorithm for low rank decomposition of tensors
of any symmetry type, from fully asymmetric to fully symmetric. 
It recovers the decomposition one summand at a time via the higher-order power method. 
This approach is known to fail in general: there need not be a relationship between the summands of a decomposition and the (partially symmetric) singular vector tuples (pSVTs) of the tensor. Our approach overcomes this problem by transforming the input to a tensor with orthonormal slices, via orthogonalization of a flattening.
The summands of the decomposition of the original tensor can be recovered from the pSVTs of this new transformed tensor. 
We introduce a shifted power method for computing pSVTs and prove its global convergence. 
Numerical experiments 
demonstrate that our algorithm achieves higher accuracy and faster runtime than existing methods.
\end{abstract}

\maketitle

\section{Introduction}\label{sec: intro}

A tensor is a multi-indexed array of numbers. 
The space of real tensors of format $m_1 \times \cdots \times m_d$ is denoted by $\mathbb{R}^{m_1} \otimes \cdots \otimes \mathbb{R}^{m_d}$. The number of indices $d$ is called the order of the tensor. We assume that $m_i\geq 2$ for all $i=1,\ldots,d$.
A tensor is \textit{partially symmetric} if its entries are unchanged under certain permutations of indices, as follows.
Let $\mathfrak{S}_d$ denote the symmetry group on $d$ letters and consider a tensor $\sT \in (\mathbb{R}^{m_1})^{\otimes d_1} \otimes \cdots \otimes (\mathbb{R}^{m_\ell})^{\otimes d_\ell}$. 
We say $\sT$ has partial symmetry of type $(d_1,\dots,d_\ell)$ if
\[
\sT_{i^1_1,\dots,i^1_{d_1},\dots,i^\ell_1,\dots,i^\ell_{d_\ell}} = 
\sT_{i^1_{\sigma_1(1)},\dots,i^1_{\sigma_1(d_1)},\dots,i^\ell_{\sigma_\ell(1)},\dots,i^\ell_{\sigma_\ell(d_\ell)} }
\]
for all $(\sigma_1,\ldots,\sigma_\ell) \in \mathfrak{S}_{d_1} \times \cdots \times \mathfrak{S}_{d_\ell}$; that is, for all permutations that act within each block of indices.  
We denote the space of tensors with this symmetry type by
\(
S^{d_1}(\mathbb{R}^{m_1}) \otimes \cdots \otimes S^{d_\ell}(\mathbb{R}^{m_\ell}) \).
For example $\sT \in S^2(\mathbb{R}^m) \otimes \mathbb{R}^n$ has symmetry type $(2,1)$, which means $\sT_{ijk} = \sT_{jik}$ for all $i, j \in [m]$. 
A tensor is (fully) symmetric if $\ell=1$ and  (fully) asymmetric if $d_1=\cdots=d_\ell=1$.

A CP decomposition writes a tensor as a sum of outer products of vectors.
An order three tensor 
$\sT \in \mathbb{R}^{m_1} \otimes \mathbb{R}^{m_2} \otimes \mathbb{R}^{m_3}$ 
has a decomposition
\(
\sT = \sum_{i=1}^r \mathbf{a}_i \otimes \mathbf{b}_i \otimes \mathbf{c}_i\),
where $\mathbf{a}_i \in \mathbb{R}^{m_1}$, $\mathbf{b}_i \in \mathbb{R}^{m_2}$, and 
$\mathbf{c}_i \in \mathbb{R}^{m_3}$, for sufficiently large $r$. 
We normalize the vectors to lie on the unit sphere 
and include a scalar coefficient $\lambda_i$ to write:
\[
\sT = \sum_{i=1}^r \lambda_i \, 
\mathbf{a}_i \otimes \mathbf{b}_i \otimes \mathbf{c}_i.
\]
The smallest $r$ for which such a decomposition holds is the rank of $\sT$. Each summand is a rank-one tensor.
CP decompositions recover interpretable structure, for example, they fit observed data to mixture models \cite{anandkumar2014tensor} and encode the complexity of linear operators \cite{landsberg2017geometry}.

For a partially symmetric tensor, it is often desirable to compute a decomposition where each rank-one term inherits the partial symmetry. This is called a \emph{partially symmetric CP decomposition}.
It writes the tensor as a sum of rank-one partially symmetric tensors, e.g. $\lambda
(\fv^{(1)})^{\otimes d_1} \otimes (\fv^{(2)})^{\otimes d_2} \otimes \cdots \otimes (\fv^{(\ell)})^{\otimes d_\ell}$. 
Enforcing partial symmetry uses fewer parameters and gives  decompositions that are easier to interpret.
Without enforcing partial symmetry in the summands,  CP decomposition  generally fails to preserve  partial symmetry.

Partially symmetric tensors and their decompositions arise across a range of settings.
In statistics and signal processing, covariance matrices, higher-order moments, and cumulants are symmetric tensors that encode distributional information and are central to PCA, Mixture Models, and ICA \cite{anandkumar2014tensor,comon1994independent,pereira2022tensor,zhang2023moment}. When two datasets are analyzed jointly, such as in contrastive PCA, contrastive ICA, and multi-context PCA~\cite{abid2018exploring,wang2024contrastive,wang2026multi},
it yields partially symmetric structure.
In elastic material analysis, the fourth-order stiffness (or compliance) tensor $\mathbscr{C}$ encodes a map from symmetric strain to symmetric stress. Physical and energetic considerations impose the symmetries $\mathbscr{C}_{ijkl}=\mathbscr{C}_{jikl}=\mathbscr{C}_{ijlk}=\mathbscr{C}_{klij}$. This tensor encodes the constitutive law of the material, and its $M$-eigenvectors, critical points of $f(\mathbf{x},\mathbf{y})=\mathbscr{C}(\mathbf{x},\mathbf{y},\mathbf{x},\mathbf{y})$, identify directional wave speeds and certify strong ellipticity \cite{han2009conditions,li2019m}. 
In quantum many-body systems, the Hilbert space (space of wave functions) of $N$ identical bosons on a $d$-level lies in $S^{N}(\CC^{d})$. 
With two distinguishable bosonic species, the Hilbert space is $S^{N_1}(\CC^{d_1})\otimes S^{N_2}(\CC^{d_2})$,  symmetric within but not across species \cite{folland2021quantum,doherty2002distinguishing}. 
In algebraic geometry, the generic rank and identifiability of partially symmetric tensors are active topics \cite{abo2010non,abo2024non,johnston2023computing}. 
Partially symmetric tensors also appear in computer vision~\cite{miao2024tensor}, 
and in training polynomial neural networks \cite{usevich2025identifiability}.
In many of these applications, preservation of the partial symmetry in the decomposition is needed to best interpret the rank-one components.

Many algorithms exist for CP decomposition. They include optimization methods, such as alternating least squares (ALS), nonlinear least squares (NLS), and unconstrained nonlinear optimization (MINF), which minimize the reconstruction error~\cite{tensorlab3.0}.  Other methods are based on simultaneous diagonalization approaches, including Jennrich’s algorithm~\cite{harshman1970foundations}, FFDIAG~\cite{ziehe2004fast}, Jacobi~\cite{cardoso1996jacobi}, QRJ1D~\cite{afsari2006simple}, and SGSD~\cite{de2004computation}, for order three tensors.

In this paper, we introduce a novel algorithm for partially symmetric tensor decomposition,  which we call the \emph{Multi-Subspace Power Method} (MSPM).
It generalizes the subspace power method (SPM)~\cite{kileel2021landscape,kileel2025subspace} from symmetric tensors to the general setting.
Unlike most CP decomposition algorithms, it applies to tensors with any partial symmetry pattern and recovers the rank-one components in the partially symmetric decomposition one at a time. To do this, it uses the partially symmetric higher-order power method (PS-HOPM). Finding components one by one avoids the ill-posedness of rank-$r$ tensor approximation~\cite{de2008tensor} (since the set of rank-one tensors is closed) and simplifies the optimization landscape, reducing sensitivity to initialization and local minima.
It is known that computing a tensor decomposition one term at a time fails in general~\cite{draisma2018best,horobect2025does,ribot2024decomposing,stegeman2010subtracting,vannieuwenhoven2014generic}. MSPM gets around this problem by transforming the input tensor to have orthonormal last slices, via a general linear change of basis. 

The key to MSPM is that tensors with orthonormal slices and sufficiently low rank have CP decompositions in which the summands correspond to the partially symmetric singular vector tuples (pSVTs), see Table~\ref{tab:tensor psvt}.
A decomposition of the transformed tensor is a transformation of a decomposition of the original tensor. 
However, best rank $r$ approximations
of the original tensor and the transformed tensor need not relate under this transformation, 
since rank $r$ approximation need not be equivariant under a non-orthogonal linear transformation.
Our numerical experiments demonstrate that the change of basis is helpful theoretically without being harmful numerically.
Tensor decomposition then reduces to computing singular vector tuples.
See Table~\ref{tab:max_rank} for the ranks supported by MSPM.

\begin{table}[tbp]
\centering
\footnotesize
\begin{tabular}{p{4.2cm} p{8cm} p{3.2cm}}
\toprule
Tensor format & Summands of CP decomposition vs. pSVTs & Model dimension, for $(\RR^{n})^{\otimes 3}$\\
\toprule
Order two (matrix) &  The SVD is a decomposition where the summands are singular vectors. & $n^2$ \newline (in~$(\RR^n)^{\otimes 2}$)\\
\midrule 
Odeco &  Each summand is a pSVT. The singular values are the coefficients~\cite{robeva2017singular}. & $3{n\choose 2}+n$\\
\midrule 
Low-rank with orthonormal last slices (rank equal to the last dimension)
& Each summand is a pSVT with singular value one; the vectors coincide in all but the last mode (Theorem~\ref{thm:decomp and singular vector}). & $3n^2-2n-{n+1\choose 2}$\\
\midrule 
Low-rank (rank equal to the rank of a flattening) & The tensor can be transformed to one with orthonormal slices. & $3n^2-2n$\\
\midrule 
Orthonormal last slices (no rank restriction) & No correspondence between summands and pSVTs, in general. & $n^3 - {n+1\choose 2}$ \\ \midrule 
General (rank higher than the rank of any flattening) & No correspondence between summands and pSVTs; subtracting a rank-one term from pSVT may increase the rank \cite{stegeman2010subtracting}. Tensors may lie in the span of the pSVTs, but such a decomposition need not be minimal \cite{draisma2018best}. & $n^3$\\
\bottomrule
\end{tabular}
\caption{The table summarizes the relationship between CP summands and pSVTs across tensor families of increasing generality. 
The model dimensions are computed by parameter counting.
The singular vector tuples of a tensor do not in general correspond to the summands in a minimal CP decomposition. This work uncovers new families of tensors for which a correspondence holds. 
Above we equate pSVTs $(\fv^{(1)}, \ldots, \fv^{(\ell)})$ of a tensor in $S^{d_1}( \RR^{m_1}) \otimes \cdots \otimes S^{d_\ell}(\RR^{m_\ell})$ with their corresponding rank one partially symmetric tensors 
$(\fv^{(1)})^{\otimes d_1} \otimes (\fv^{(2)})^{\otimes d_2} \otimes \cdots \otimes (\fv^{(\ell)})^{\otimes d_\ell}$.}
\label{tab:tensor psvt}
\end{table}

\begin{table}[tbp]
\centering
\begin{tabular}{cc}
\toprule
Tensor Space & Maximal Rank \\
\midrule
$(\mathbb{R}^n)^{\otimes 3}$ &  $n$  \\
$S^2 (\RR^n) \otimes \RR^n$ &  $n$\\
$S^3 (\RR^n) \otimes \RR^n$ &  
${n+1 \choose 2}$
\\
$S^2 (\RR^n) \otimes S^2 (\RR^n)$ &  
$(n-1)^2$
\\
$(\mathbb{R}^n)^{\otimes 4}$ & 
$(n-1)^2$
\\
$S^4 (\RR^n) \otimes \RR^n$ &  
$(n-1)^2$
\\
$S^3 (\RR^n) \otimes S^2 (\RR^n)$ &  
$n^2$
\\
\bottomrule
\end{tabular}
\caption{Maximum rank allowed by MSPM, assuming $n \geq 5$.}
\label{tab:max_rank}
\end{table}

To compute partially symmetric singular vector tuples in MSPM, we introduce the partially symmetric Higher-Order Power Method (PS-HOPM),  which generalizes both the Higher-Order Power Method \cite{de1995higher} and the shifted power method for symmetric tensors \cite{kolda2011shifted}.
We prove global convergence and local linear convergence of PS-HOPM in a special case.
We demonstrate through numerical experiments on noisy tensors that our overall decomposition algorithm achieves higher accuracy and faster run times than existing approaches.
One can work with noisy tensors by truncating the SVD in the first step of the algorithm.
An implementation of MSPM is available at 
\url{https://github.com/joaompereira/SPM}.

Our MSPM algorithm has four parts:
\begin{enumerate}
\item Form a tensor with last slices orthonormal:
   Given $\sT\in S^{d_1}(\RR^{m_1})\otimes\cdots\otimes S^{d_\ell}(\RR^{m_\ell})$, compute its flattening with rows of symmetry type $(f_1,\ldots,f_\ell)$, where $0\leq f_i \leq d_i$. Compute an orthonormal basis for its column space (e.g. via SVD). Stack the basis vectors to obtain a tensor $\sT^{(f)}$ whose last slices are orthonormal. See Figure~\ref{fig:flattening_pipeline}.

\item Rank-one recovery: 
    Apply PS-HOPM (Section~\ref{sec: power method}) to $\sT^{(f)}$ to compute partially symmetric singular vector tuples (pSVTs) with singular value one. Each yields a partially symmetric rank-one component $\bigotimes_{i=1}^\ell{(\fv^{(i)})}^{\otimes f_i}$.

\item Completion and deflation: 
   Construct the complementary tensor $\bigotimes_{i=1}^\ell{(\fv^{(i)})}^{\otimes (d_i-f_i)}$ (extra flattenings are needed when some $f_i=0$), estimate its scalar coefficient $\lambda$, and subtract $\lambda \bigotimes_{i=1}^\ell{(\fv^{(i)})}^{\otimes d_i}$ from the current residual tensor.  

\item Iteration until full decomposition:
   Repeat the rank-one recovery step and the completion and deflation step until all $r$ components are recovered.
\end{enumerate}

\begin{figure}[tbp]
    \centering
    \includegraphics[width=0.9\textwidth]{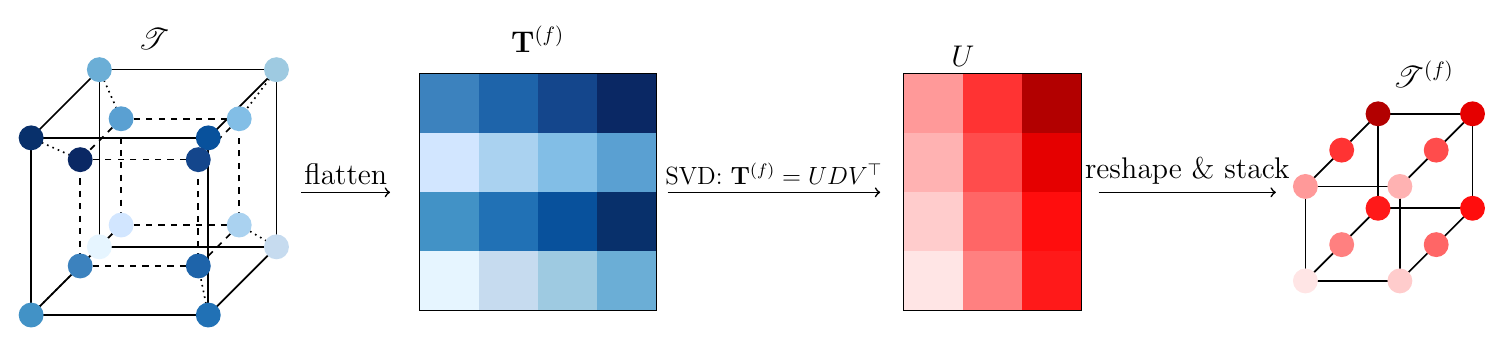} 
    \caption{
        Building a tensor with last slices orthonormal from a general tensor.
        A fourth-order tensor $\sT$ is represented as a 4D hypercube. 
        Flattening with $f=(1,1,0,0)$ combines the first two directions to be columns in the flattening matrix $\mathbf{T}^{(f)}$. 
        The SVD yields the left singular vectors $U$. 
        The columns of $U$ are reshaped into tensors and stacked to form the tensor $\sT^{(f)}$. 
    }
    \label{fig:flattening_pipeline}
\end{figure}

The rest of this paper is organized as follows.  
In Section~\ref{sec:background} we give background on partially symmetric tensors, their flattenings, and singular vector tuples. 
Section~\ref{sec:decomposition via svt} presents our first main result, showing how tensors with orthonormal slices can be decomposed using their singular vector tuples. 
Section~\ref{sec: power method} describes the \emph{Partially Symmetric Higher-Order Power Method} and establishes its global convergence and local linear convergence. 
Section~\ref{subsec: tensor transformation} provides rank bounds that guarantee uniqueness. 
Sections~\ref{subsec:completion}, \ref{subsec:deflation} develops the factor completion and deflation procedures needed for full decompositions. 
Section~\ref{sec:experiments} reports numerical experiments that compare our method with existing approaches. 
Section~\ref{sec:conclusion} concludes the paper.

\section{Background}\label{sec:background}

A partially symmetric tensor in $S^{d_1}(\mathbb{R}^{m_1}) \otimes \cdots \otimes S^{d_\ell}(\mathbb{R}^{m_\ell})$ has a decomposition
\[
\sT = \sum_{i=1}^r \lambda_i
(\fv_i^{(1)})^{\otimes d_1} \otimes (\fv_i^{(2)})^{\otimes d_2} \otimes \cdots \otimes (\fv_i^{(\ell)})^{\otimes d_\ell},
\]
where $\fv_i^{(j)}$ lies in the unit sphere $\mathbb{S}^{m_j-1}$ for $j=1,\dots,\ell$
and $r$ is sufficiently large.  
Let $\|\cdot\|$ denote the Frobenius norm (the square root of the sum of the squares of the entries).  
In practice, computing a partially symmetric decomposition amounts to solving
\[
\min_{\{\fv_i^{(j)},\lambda_i\}} \;
\left\|
\sT - \sum_{i=1}^r \lambda_i
\left(\fv_i^{(1)}\right)^{\otimes d_1} \otimes \left(\fv_i^{(2)}\right)^{\otimes d_2} \otimes \cdots \otimes \left(\fv_i^{(\ell)}\right)^{\otimes d_\ell}
\right\|.
\]

Tensor decomposition 
is challenging, both theoretically and practically. It is in general NP hard~\cite{haastad1990tensor,hillar2013most}. There are numerical challenges that arise from the fact that the set of tensors of rank at most $r$ is not closed~\cite{de2008tensor} and the  optimization landscape is non-convex~\cite{ge2017optimization}.

\begin{definition}[Contraction]
		Let $\sT$ be a tensor in  
		$S^{d_1}(\mathbb{R}^{m_1}) \otimes \cdots \otimes S^{d_\ell}(\mathbb{R}^{m_\ell})$.
		For vectors $\fv^{(i)} \in \mathbb{R}^{m_i}$ and non-negative integers $f_1\le d_1, \dots,f_\ell \le d_\ell$, the contraction of $\sT$ with 
		$\bigotimes_{i=1}^\ell \left(\fv^{(i)}\right)^{\otimes f_i}$
		produces a tensor in  
		$S^{d_1-f_1}(\mathbb{R}^{m_1}) \otimes \cdots \otimes S^{d_\ell-f_\ell}(\mathbb{R}^{m_\ell})$.
		We denote it by  
		\[ \sT \cdot \bigotimes_{i=1}^\ell \left(\fv^{(i)}\right)^{\otimes f_i} .
		\]
		Its entries are  
		\[
		\left(\sT \cdot \bigotimes_{i=1}^\ell \left(\fv^{(i)}\right)^{\otimes f_i}\right)_{\mathbf{k}_1\cdots \mathbf{k}_\ell}
		=
		\sum_{\mathbf{j}_1 \in [m_1]^{f_1}} \cdots \sum_{\mathbf{j}_\ell \in [m_\ell]^{f_\ell}}
		\sT_{(\mathbf{j}_1, \mathbf{k}_1) \cdots (\mathbf{j}_\ell, \mathbf{k}_\ell)} \prod_{i=1}^\ell (\fv^{(i)})^{\otimes f_i}_{\mathbf{j}_i},
		\]
		where $[m] = \{1,\dots,m\}$ and $\mathbf{k}_i \in [m_i]^{d_i - f_i}$.
	\end{definition}
	
	Let $\delta_{i,j}$ be $1$ if $i=j$ and $0$ otherwise.  
    For $f_i = d_i - \delta_{i,j}$, contraction yields a vector in $\RR^{m_j}$.
	
	\begin{definition}[Partially symmetric singular vector tuple (pSVT), see~\cite{friedland2014number}]
		Fix a tensor 
		\(
		\sT \in \bigotimes_{i=1}^\ell S^{d_i}(\mathbb{R}^{m_i})
		\).
		A \emph{partially symmetric singular vector tuple} is a tuple of vectors
		\[
		(\fv^{(1)}, \ldots, \fv^{(\ell)}) \in \mathbb{S}^{m_1 - 1} \times \cdots \times \mathbb{S}^{m_\ell - 1}
		\]
		such that, for $j=1,\dots,\ell$, the following equation holds:
		\begin{equation}\label{eq:svt}
			\sT \cdot \bigotimes_{i=1}^\ell \left(\fv^{(i)}\right)^{\otimes d_i - \delta_{ij}} = \sigma \fv^{(j)},
		\end{equation}
		The scalar $\sigma \in \mathbb{R}$ is called the \emph{singular value}, which can be computed explicitly as the full contraction
\(
\sigma \;=\; \sT \cdot \bigotimes_{i=1}^\ell \bigl(\fv^{(i)}\bigr)^{\otimes d_i}.
\)
    
		\end{definition}

 Throughout the paper, when we talk about pSVTs, we mean 
real vector tuples.
They specialize to usual singular vector tuples (SVTs) \cite{lim2005singular} for when $d_i = 1$ for all $i$, and to tensor eigenvectors when $\ell = 1$. 
The pSVT pairs $(\fv^{(1)},\ldots,\fv^{(\ell)},\sigma)$ of $\sT$ are the critical points of the rank-one approximation error
\[
\left\| \sT - \sigma \bigotimes_{i=1}^\ell (\fv^{(i)})^{\otimes d_i} \right\|^2,
\]
see \cite{lim2005singular, hu2018convergence,zhang2001rank}.
For matrices, singular values can always be assumed to be non-negative by flipping the sign of a
singular vector. For partially symmetric tensors with some $d_i$ odd, the same approach works to assume $\sigma \geq 0$. When all exponents $d_i$ are even, the singular value is invariant under sign flips of the vectors (just like for eigenvectors of matrices).

\begin{proposition}[{See~\cite[Theorem 12]{friedland2014number}}]
A generic partially symmetric tensor $\sT$ has finitely 
many pSVTs.
\end{proposition}

A tensor can be flattened (or unfolded) to a matrix in various ways. 

\begin{definition}[Flattenings of a partially symmetric tensor]
\label{def:flattening}
Fix $\sT \in \bigotimes_{i=1}^\ell S^{d_i}(\mathbb{R}^{m_i})$, and an integer tuple $f = (f_1, \ldots, f_\ell)$ with $0 \leq f_i \leq d_i$. 
The $f$-flattening of $\sT$, denoted $\mathbf{T}^{(f)}$, is a matrix with rows and columns indexed by entries of partially symmetric tensors in
\[
\bigotimes_{i=1}^\ell S^{f_i}(\mathbb{R}^{m_i}) \quad \text{and} \quad 
\bigotimes_{i=1}^\ell S^{d_i - f_i}(\mathbb{R}^{m_i}),
\]
respectively.
When the tuple $f$ is $(d_1,\ldots,d_\ell)$, the flattening $\mathbf{T}^{(f)}$ is a vector, denoted $\operatorname{vect}(\sT)$ and called the vectorization of $\sT$.  
\end{definition}

For example, when $\sT = (\fv^{(1)})^{\otimes d_1} \otimes \cdots \otimes (\fv^{(\ell)})^{\otimes d_\ell}$, 
\[
\mathbf{T}^{(f)} = \operatorname{vect}\left( \left(\fv^{(1)}\right)^{\otimes f_1} \otimes \cdots \otimes \left(\fv^{(\ell)}\right)^{\otimes f_\ell} \right) \otimes \operatorname{vect}\left( \left(\fv^{(1)}\right)^{\otimes (d_1 - f_1)} \otimes \cdots \otimes \left(\fv^{(\ell)}\right)^{\otimes (d_\ell - f_\ell)} \right).
\]

\begin{definition}[Orthonormal slices]
\label{def:slices}
Given $\sT \in \bigotimes_{k=1}^\ell \mathbb{R}^{m_k}$, its 
$j$-{\em slices} are a collection of $m_j$ order $\ell-1$ tensors in 
\( \mathbb{R}^{m_1} \otimes \cdots \otimes \mathbb{R}^{m_{j-1}} \otimes \mathbb{R}^{m_{j+1}} \otimes \cdots \otimes \mathbb{R}^{m_\ell} \)
obtained by fixing the value of the $j$-th index of $\sT$.
We call the $\ell$-slices the \emph{last slices}. 
We say the $j$-slices of $\sT$ are \emph{orthonormal} if their vectorizations are orthonormal vectors. 
\end{definition}

We record the following linear algebra fact.
\begin{proposition}
\label{prop:orthonormalize}
    A tensor $\sT\in \bigotimes_{k=1}^\ell \RR^{m_k}$ with linearly independent last slices can be transformed to one whose last slices are orthonormal, via general linear change of basis in the last index.
\end{proposition}

\begin{proof}
Let \( \mathbf{M} \in \RR^{(m_1\cdots m_{\ell-1})\times m_\ell} \) be the flattening of \( \sT \) using the tuple~$(1,\ldots,1,0)$.
Then $\mathbf{M}$ has as columns the vectorizations of the last slices of $\sT$. The matrix $\mathbf{M}$ has full column rank $m_\ell$, by the assumption of linearly independent last slices.
Let $\mathbf{M} = QR$ be a thin QR decomposition, where $Q$ has orthonormal columns and $R$ is invertible. Then
$\mathbf{M}R^{-1} = Q.$
Applying $R^{-1}$ in the last mode of $\sT$ thus gives a tensor with orthonormal last slices.
\end{proof}

\section{A decomposition from singular vector tuples}\label{sec:decomposition via svt}

In this section, we study tensors whose slices are orthonormal. A tensor can be transformed into this form via change of basis along one index (see Proposition~\ref{prop:orthonormalize}). We show that the rank-one summands of such tensors are in one-to-one correspondence with their pSVTs, when its rank equals the number of slices. 
The pSVTs may differ from the rank-one components in their last factors.

The general linear transformation used to achieve orthonormality of the slices of the tensor transforms an associated CP decomposition. However, it will in general change the low-rank approximations of the tensor, since these are not GL-equivariant, as described in the introduction. A central numerical contribution of MSPM is that this general linear change of basis does not affect noisy best rank $r$ approximation in practice.
    The setting of this section, and of MSPM, differs from orthogonal tensor decomposition~\cite{chen2009tensor,sorensen2012canonical,wang2015orthogonal,yang2020epsilon,hu2023linear} as we do not impose orthogonality on the factors.

In the remainder of this section, therefore, we study tensors with orthonormal last slices.
We relate the CP decomposition to its pSVT.
For each pSVT with singular value one, we see that all but one of the vectors appear together in a summand of the decomposition. 

\begin{example}[A $2 \times 2 \times 2$ tensor]
\label{ex:one}
Consider the tensor 
\[
\sT = \mathbf{u} \otimes \mathbf{u} \otimes \mathbf{v} \;+\; \mathbf{v} \otimes \mathbf{v} \otimes \mathbf{u}, \quad 
\text{where}
\quad 
\mathbf{u} := \begin{pmatrix} \tfrac{1}{\sqrt{2}} \\ \tfrac{1}{\sqrt{2}} \end{pmatrix}, \quad 
\fv := \begin{pmatrix} \tfrac{1}{\sqrt{2}} \\ -\tfrac{1}{\sqrt{2}} \end{pmatrix}.
\]
The tensor $\sT$ is odeco \cite{robeva2016orthogonal,robeva2017singular}.
The third slices of $\sT$ are orthonormal.  
We compute the pSVTs of $\sT$ by solving polynomial equations~\eqref{eq:svt} in the package \texttt{HomotopyContinuation.jl} in \texttt{Julia}~\cite{breiding2018homotopycontinuation}.
The polynomial system has six complex solutions, of which four are the real pSVTs.  
Among them, the two tuples
\(
(\mathbf{u},\mathbf{u},\mathbf{v}) \) and \( (\mathbf{v},\mathbf{v},\mathbf{u}) \)
have singular value one. They are the two rank-one summands of $\sT$.  
\end{example}

\begin{example}
\label{ex:two}
    Let $\mathbf{u}$ be as in the previous example and consider
\[
\sT' = 
\mathbf{u} \otimes \mathbf{u} \otimes 
\begin{pmatrix} \tfrac{1}{\sqrt{3}} \\[0.3em] 1 \end{pmatrix}
+ 
\mathbf{e}_1 \otimes \mathbf{e}_1 \otimes 
\begin{pmatrix} -\tfrac{2}{\sqrt{3}} \\[0.3em] 0 \end{pmatrix}, \quad \text{where} \quad \fe_1 := \begin{pmatrix} 1 \\ 0 \end{pmatrix}, \fe_2 :=\begin{pmatrix} 0 \\ 1 \end{pmatrix}.
\]
The tensor $\sT'$ is \emph{not} odeco since $\fu$ is not orthogonal to $\fe_1$, but its last slices 
$$
\begin{pmatrix}
\frac{-\sqrt{3}}{2} & \frac{1}{2\sqrt{3}} \\
\frac{1}{2\sqrt{3}} & \frac{1}{2\sqrt{3}}
\end{pmatrix}, \quad
\begin{pmatrix}
\frac{1}{2} & \frac{1}{2}\\
\frac{1}{2} & \frac{1}{2}
\end{pmatrix}
$$
are orthonormal. 
The decomposition is unique, by 
Kruskal's criterion \cite{kruskal1977three}. 

We compute the pSVTs of $\sT'$ by solving polynomial equations~\eqref{eq:svt}.
This system has six complex solutions and all of them are real.
Two have singular value one:
\[
(\mathbf{e}_1,\mathbf{e}_1,\begin{pmatrix} -\tfrac{\sqrt{3}}{2} \\[0.3em] \tfrac{1}{2} \end{pmatrix}), 
\quad (\mathbf{u},\mathbf{u},\mathbf{e}_2).
\]
The first two factors of each pSVT, $\mathbf{e}_1 \otimes \mathbf{e}_1$ and $\mathbf{u} \otimes \mathbf{u}$, appear in the decomposition of $\sT'$. The third vectors in the pSVTs do not match the decomposition factors. 
\end{example}

We next prove the one-to-one correspondence, observed in Examples~\ref{ex:one} and~\ref{ex:two}, between pSVTs with singular value one and the rank-one summands in the decomposition, for a tensor with orthonormal slices.
Partially symmetric SVTs with singular value one exist tensors whose slices are orthonormal and whose rank equals the number of orthonormal slices.
The problem of decomposing tensors without orthonormal slices and of higher rank, will be addressed in the  subsequent sections.
First we recall a related result for matrices.

\begin{lemma}\label{lem: matrix with orthonormal columns}
Fix $m \geq n$. Let $\mathbf{M}\in \RR^{m\times n}$ be a matrix with orthonormal columns. The singular values of $\bM$ are 0 or 1. A vector $\fu\in \SB^{m-1}$ is a left singular vector of $\mathbf{M}$ with singular value one if and only if it lies in the column span of $\mathbf{M}$.
\end{lemma}
\begin{proof}
We have $\mathbf{M}\T \mathbf{M} = \mathbf{I}_n$, 
since $\mathbf{M}$ has orthonormal columns. Hence, the eigenvalues of $\mathbf{M}\T \mathbf{M}$ are all one, so the nonzero singular values of $\mathbf{M}$ are all one.
Suppose $(\fu,\fv)$ is a singular vector pair of $\mathbf{M}$ with singular value one. It follows that $\fu = \bM \fv$ so $\fu\in \colspan \bM$.
Conversely, let $\fu \in \colspan(\mathbf{M})$ be a unit vector. Define $\fv = \bM\T \fu$. 
Using that $\bM\bM\T$ is orthogonal projection onto $\colspan \bM$, we have 
$\bM\fv = \bM \bM\T \fu = \fu$ and $\|\fv\|^2 = \fu\T \bM \bM\T \fu = \| \fu\|^2 = 1$. So $(\fu,\fv)$ is a singular vector pair with singular value one.
\end{proof}

\begin{theorem}\label{thm:decomp and singular vector}
Fix 
\(
\sT \in 
\Bigl( \bigotimes_{i=1}^\ell S^{d_i}(\mathbb{R}^{m_i}) \Bigr) \otimes \mathbb{R}^m.
\)
Let \(
\Acal \subseteq \bigotimes_{i=1}^\ell S^{d_i}(\mathbb{R}^{m_i})
\)
be the span of the last slices of $\sT$, which we assume to be orthonormal.
Then:
\begin{enumerate}
    \item[(i)] The pSVTs of $\sT$ have singular values in  $[0,1]$. The rank-one tensors in $\bigotimes_{i=1}^\ell S^{d_i}(\mathbb{R}^{m_i})$ corresponding to pSVTs with singular value one are the rank-one tensors in $\Acal$. 
    \item[(ii)] Supposing further that \( \sT \) has rank \( m \) and decomposition
    \begin{equation}\label{eq:tensordecomp}
\sT = \sum_{i=1}^m \lambda_i \, \bigl(\fv^{(1)}_i\bigr)^{\otimes d_1} \otimes \cdots \otimes \bigl(\fv^{(\ell)}_i\bigr)^{\otimes d_\ell} \otimes \fu_i,
\end{equation}
then for each summand in the decomposition  there exists $\fw_i \in \mathbb{S}^{m-1}$ such that the tuple
\(
(\fv^{(1)}_i, \ldots, \fv^{(\ell)}_i, \fw_i)
\)
is a pSVT of $\sT$ with singular value one. A linear transformation computed from $\sT$ maps $\{\fw_i\}_{i=1}^m$ to $\{\fu_i\}_{i=1}^m$. 
\end{enumerate}
\end{theorem}

\begin{proof}
(i) Let \( \mathbf{M} \) denote the flattening of \( \sT \) using the tuple~$(d_1,\ldots,d_\ell,0)$.
It has $M = \prod_{i=1}^\ell m_i^{d_i}$ rows and $m$ columns. The columns are orthonormal, since the last slices of $\sT$ are orthonormal. 
Their span is $\Acal$.
Let $(\fv^{(1)}, \ldots, \fv^{(\ell)}, \fw)$ be a pSVT of $\sT$ with singular value $\sigma$. We can assume $\sigma$ is non-negative, after changing the sign of $\fw$ if necessary.

The maximum singular value of $\sT$ is the maximum value of 
$\sT \cdot \bigotimes_{i=1}^\ell (\fv^{(i)})^{\otimes d_i}\otimes \fu$ 
over $\fv^{(i)}\in \SB^{m_i-1}$ and $\fu\in \SB^{m-1}$. 
Define  $\sR = \bigotimes_{i=1}^\ell (\fv^{(i)})^{\otimes d_i}$.
Then $\Vect(\sR)\in \SB^{M-1}$ and 
$$\sT \cdot \bigotimes_{i=1}^\ell (\fv^{(i)})^{\otimes d_i}\otimes \fu = \Vect(\sR)\T \bM \fu.$$ 
The matrix $\bM$ has maximum singular value one, by Lemma \ref{lem: matrix with orthonormal columns}.  Thus, the singular values of $\sT$ lie in $[0,1]$. 

Suppose $(\fv^{(1)}, \ldots, \fv^{(\ell)}, \fw)$ is a pSVT with singular value one. 
Then $(\Vect(\sR), \fw)$ is a singular vector pair of $\bM$ with singular value one. We obtain $\Vect(\sR) = \bM \fw\in \colspan(\bM)$, by Lemma \ref{lem: matrix with orthonormal columns}. Hence $\sR\in \Acal$. 
For the converse, suppose $\sR\in \Acal$. Then $(\Vect(\sR),\fw)$ is a singular vector pair of $\bM$ with singular value one, again by Lemma \ref{lem: matrix with orthonormal columns}.
This implies 
\[
\sT(*,\ldots,*,\fw)=\sR, \quad \text{and} \quad \sT \cdot \sR = \fw.
\]
Therefore, $(\fv^{(1)},\ldots,\fv^{(\ell)},\fw)$ is a pSVT with singular value one, since $(\fv^{(1)},\ldots,\fv^{(\ell)})$ is a pSVT of $\sR_s$ and $\fw = \sT\cdot \bigotimes_{i=1}^\ell (\fv^{(i)})^{\otimes d_i}$. This concludes the proof of (i).

(ii) Suppose $\rank\sT=m$ and that \eqref{eq:tensordecomp} holds. 
We have 
$\sT = \sum_{i=1}^m \lambda_i \,\sR^{(i)} \otimes \fu_i$, where 
$\sR^{(i)} = \bigotimes_{j=1}^\ell (\fv_i^{(j)})^{\otimes d_j}$. 
The vectors \( \operatorname{vect}(\sR^{(1)}), \ldots, \operatorname{vect}(\sR^{(m)})\) are a basis of the linear space $\Acal$, since $\rank \sT = m$. 
Hence 
$\sR^{(i)}\in \Acal$ and 
$(\fv^{(1)}_i,\ldots,\fv^{(\ell)}_i,\fw_i)$ is a pSVT with singular value one, where $\fw_i = \sT\cdot \sR^{(i)}$.

We now relate $\{\fu_i\}_{i=1}^m$ and $\{\fw_i\}_{i=1}^m$.
Let $\mathbf{U},\mathbf{W}\in \RR^{m\times m}$ be $\mathbf{U} = \begin{pmatrix}
-\fu_1- \\ \vdots \\ -\fu_m-
\end{pmatrix}, \mathbf{W} = \begin{pmatrix}
-\fw_1- \\ \vdots \\ -\fw_m-
\end{pmatrix}$. 
The decomposition of $\sT$ yields a decomposition of its flattening
$$\mathbf{M}=\sum_{i=1}^m \lambda_i \,\Vect(\sR^{(i)})\otimes \fu_i
= \mathbf{M}_{\sR}\T\,\mathbf{\Lambda}\,\mathbf{U}\in \RR^{M\times m},$$
where 
$\Lambda = \Diag(\lambda_1,\ldots,\lambda_m)$ and $\bM_\sR = \begin{pmatrix}
-\Vect(\sR^{(1)})- \\
\vdots \\
-\Vect(\sR^{(m)})-
\end{pmatrix}\in \RR^{m\times M}$.
By construction, we have $\sT \cdot \sR^{(i)} = \fw_i$, so 
$\mathbf{M}_\sR \mathbf{M} = \mathbf{W}$.
Combining with $\mathbf{M}=\mathbf{M}_{\sR}\T\mathbf{\Lambda} \mathbf{U}$ yields
\[
\mathbf{W} = \mathbf{M}_{\sR}\mathbf{M}_{\sR} \T\,\mathbf{\Lambda}\,\mathbf{U},
\quad\text{so}\quad
\mathbf{U} = \mathbf{\Lambda}^{-1}\bigl(\mathbf{M}_{\sR} \mathbf{M}_{\sR}\T\bigr)^{-1}\mathbf{W}. \qedhere
\] 
\end{proof}

\begin{remark}
The proof of Theorem~\ref{thm:decomp and singular vector} implies a  stronger statement: an SVT of $\sT$ with singular value one is a pSVT of $\sT$ (i.e., an SVT with the partial symmetry of $\sT$).
\end{remark}

\begin{remark}
In the special case that the tensors $\{ \sR^{(i)} \}_{i=1}^m$ are orthonormal,
the decomposition of $\sT$ consists of rank-one components that are pSVTs, since $\mathbf{M}_{\sR} \mathbf{M}_{\sR}\T=\mathbf{I}$ and thus $\mathbf{U}=\mathbf{W}, \mathbf{\Lambda}=\mathbf{I}$.  
The columns of $\bM$ form an orthonormal basis of $\Acal$, so
the matrix $\mathbf{M}\mathbf{M}\T$ is the projection matrix onto $\Acal$ and $\mathbf{M}\mathbf{M}\T\bM_{\sR}\T = \bM_{\sR}\T$.
The vectors $\{\fw_1,\ldots,\fw_m\}$ are also orthonormal since 
$\mathbf{W}\mathbf{W}\T = \mathbf{M}_\sR \mathbf{M}\mathbf{M}\T\mathbf{M}_\sR\T = \mathbf{M}_\sR \bM_\sR\T= \mathbf{I}$. 
\end{remark}

\begin{remark}Theorem~\ref{thm:decomp and singular vector} requires that the slices of $\sT$ are orthonormal, it is not enough to be orthogonal, since the bound on the singular values would break. For Theorem~\ref{thm:decomp and singular vector}(ii), if the rank exceeds the number of slices, we still have that a tensor in $\mathcal{A}$ can be written as a linear combination of the rank one tensors $\sR^{(i)}$, but lose the containment $\sR^{(i)} \in \mathcal{A}$, cf.~\cite{wang2023lower}. The rank is at least $m$, the number of slices, since the slices are orthonormal hence linearly independent, and the rank of $\sT$ is at least the rank of its flattening.
\end{remark}

\section{A partially symmetric power method}\label{sec: power method}

In this section we consider a power method for computing pSVTs of  partially symmetric tensors, which unifies two classical algorithms: the Higher-Order Power Method (HOPM)~\cite{de1995higher} and the shifted power method for symmetric tensors~\cite{kolda2011shifted}.  
We call it the \emph{partially symmetric higher order power method} (PS-HOPM). 
PS-HOPM is described in Algorithm~\ref{alg:psspm}, with the implementation of power method iteration described in Algorithm \ref{alg:pmi_eff}. 
It is applicable to general partially symmetric tensors, and may be of independent interest.

\begin{figure}[htpb]
   \renewcommand\theContinuedFloat{\alph{ContinuedFloat}} 
	\algrenewcommand\algorithmicindent{0.65em}
	\begin{minipage}[t]{.47\textwidth}
	\begin{algorithm}[H]
	\caption{Partially Symmetric Higher Order Power Method (PS-HOPM)}
	\label{alg:psspm}
	\begin{algorithmic}
	\Require Partially symmetric tensor 
	$$
	\sT \in \bigotimes_{i=1}^\ell S^{a_i}(\mathbb{R}^{m_i})$$
	\Require Shift parameters $\gamma_1,\ldots,\gamma_\ell \geq 0$, stopping criterion
	\Ensure A pSVT $\fv^{(1\,:\,\ell)}=(\fv^{(1)},\ldots,\fv^{(\ell)})$ with singular value $\sigma$ \bigskip
	
	\For{$i = 1,\ldots,\ell$}
	  \State $\fv^{(i)} \gets$ random unit vector in $\mathbb{S}^{m_i - 1}$
	\EndFor
	\Repeat
	\State $\displaystyle \fv^{(1\,:\,\ell)} \gets \text{\textsc{PMI}}(\sT; \fv^{(1\,:\,\ell)})$

	\Until{stopping criterion is met}
	
	\State $\sigma \gets 
	  \sT \cdot \bigotimes_{i=1}^\ell \bigl(\fv^{(i)}\bigr)^{\otimes a_i}$
	\end{algorithmic}
	\end{algorithm}
	\end{minipage}
	\hfill
	\begin{minipage}[t]{.51\textwidth}
	\begin{algorithm}[H]
		\ContinuedFloat 
		\caption{Power Method Iteration (PMI)}\label{alg:pmi_eff}
		\begin{algorithmic}
			\Require Partially symmetric tensor 
			\[
			\sT \in \bigotimes_{i=1}^\ell S^{a_i}(\mathbb{R}^{m_i})
			\]
			\Require Vectors $\fv^{(1\,:\,\ell)} = (\fv^{(1)},\ldots,\fv^{(\ell)})$ with unit norm
			\Require Shift parameters $\gamma_1,\ldots,\gamma_\ell \geq 0$
			\Ensure Vectors $\fw^{(1\,:\,\ell)} = (\fw^{(1)},\ldots,\fw^{(\ell)})$ satisfying \eqref{eq:PMI}
			\bigskip
			
			\If{$\ell = 1$}
			\State $\tilde{\fv}^{(1)} \gets 
			\sT\cdot \bigl(\fv^{(1)}\bigr)^{\otimes a_1-1}
			\;+\; \gamma_i \fv^{(1)}$
			\If{$\|\tilde{\fv}^{(1)}\| > 0$}
			\State $\fv^{(1)} \gets \tilde{\fv}^{(1)} / \|\tilde{\fv}^{(1)}\|$
			\EndIf
			\Else
			\State $k\gets \lfloor\frac\ell2\rfloor$
			\State $\mathbscr{U}\gets \sT\cdot \textstyle \bigotimes_{j=k+1}^\ell \bigl(\fv^{(j)}\bigr)^{\otimes a_j}$
			\State 
			$\fw^{(1\,:\,k)} \gets \text{\text{PMI}}(\mathbscr{U}; \fv^{(1\,:\,k)})$
			\State $\mathbscr{W}\gets \sT\cdot \textstyle \bigotimes_{j=1}^k \bigl(\fv^{(j)}\bigr)^{\otimes a_j}$
			\State
			$\fw^{(k+1\,:\,\ell)} \gets \text{\text{PMI}}(\mathbscr{W};  \fv^{(k+1\,:\,\ell)})$
			\EndIf
		\end{algorithmic}
	\end{algorithm}
\end{minipage}
\end{figure}

In the fully asymmetric case, with $\gamma_1=\cdots=\gamma_\ell=0$, Algorithm~\ref{alg:psspm} reduces to the classical HOPM, known to converge globally to a second-order critical point of $(\fv_1,\ldots,\fv_\ell)\mapsto \sT \cdot \bigotimes_{i=1}^\ell \fv_i^{a_i}$, see~\cite{uschmajew2014new}.  
In the fully symmetric case, it specializes to the shifted power method, which also has global convergence to a second-order critical point~\cite{kileel2025subspace}. 
A related work in the partially symmetric setting is the recent paper~\cite{ni2025alternating}, whose alternating algorithm is mathematically equivalent to the update scheme in Algorithm \ref{alg:pmi_eff}. 
The authors show monotonic increase of the function values, but do not address convergence of the iterates, the choice of shifts, or local convergence rates. We address these three aspects in the following subsections.

Mathematically, Algorithm \ref{alg:pmi_eff} is equivalent to the partially symmetric shifted power method iteration, defined as follows
\begin{equation}\label{eq:PMI}
\fw^{(i)} =\frac{\sT\cdot \bigotimes_{j=1}^\ell \bigl(\fv^{(j)}\bigr)^{\otimes a_j -\delta_{i,j}}}{\left\|\sT\cdot \bigotimes_{j=1}^\ell \bigl(\fv^{(j)}\bigr)^{\otimes a_j -\delta_{i,j}}\right\|},\quad i=1,\dots,\ell
\end{equation}
Evaluating \eqref{eq:PMI} directly involves calculating $\ell$ tensor contractions with $\sT$, which requires $\ell M$ total multiplications, where $M= \prod_{i=1}^\ell m_i^{a_i}$ is the number of entries of $\sT$. Alternatively, our implementation uses the fact that, if $i\le k$, $\sT\cdot \bigotimes_{j=1}^\ell \bigl(\fv^{(j)}\bigr)^{\otimes a_j -\delta_{i,j}} = \mathbscr{U}\cdot \bigotimes_{j=1}^k \bigl(\fv^{(j)}\bigr)^{\otimes a_j -\delta_{i,j}}$, with $\mathbscr{U}$ defined as in Algorithm~\ref{alg:pmi_eff}. By pre-calculating $\mathbscr{U}$ and $\mathbscr{W}$, we reduce the computational cost from $\ell M$ to $2 M + O(M_1) + O(M_2)$, where $M_1= \prod_{i=1}^k m_i^{a_i}$ and
$M_2= \prod_{i=k+1}^\ell m_i^{a_i}$ are the number of entries of $\mathbscr{W}$ and $\mathbscr{U}$, respectively. This results in a computational speed-up of approximately $\ell/2$. Similar techniques for speeding up computations, relying on pre-calculating quantities that are shared between modes, have been proposed in the context of CP tensor decomposition \cite{phan2013fast,kaya2019computing}.

\subsection{Global convergence}

We show that Algorithm \ref{alg:psspm} has global convergence to a pSVT for suitable shifts $\gamma_1,\ldots,\gamma_\ell$. 
We refer to the variables at one symmetric factor $S^{a_i}(\RR^{m_i})$ as a block. A block update means we update the vector  in one block, fixing  the other factors. Let $\| \cdot \|_2$ denote the spectral norm of a tensor, the largest absolute value of its singular values.

\begin{theorem}\label{thm:power_method_convergence}
Fix 
$\sT \in \bigotimes_{i=1}^\ell S^{a_i}(\mathbb{R}^{m_i}).$
Define the function 
$F: \SB^{m_1-1}\times \cdots \times \SB^{m_\ell-1} \to  \RR$
by
\begin{equation}\label{eq:def_F}
F(\fv^{(1)},\ldots,\fv^{(\ell)}) 
= \Bigl\langle \sT, \bigotimes_{i=1}^\ell \bigl(\fv^{(i)}\bigr)^{\otimes a_i} \Bigr\rangle
+ \sum_{i=1}^\ell \gamma_i \|\fv^{(i)}\|^{a_i}.
\end{equation}
Then the pSVTs of $\sT$ are the first-order critical points of $F$. 
There exists $\gamma_i \geq 0$ such that 
the restriction of $F$ to the $i$-th block 
is convex over $\RR^{m_i}$ for all values of the other blocks.  
A sufficient choice is 
$\gamma_i \geq (a_i-1)\|\sT\|_2$.
If all block restrictions are convex over $\RR^{m_i}$, then Algorithm \ref{alg:psspm} applied to $\sT$ converges globally to a pSVT of $\sT$. 
The convergence rate is at least algebraic.
\end{theorem}

\begin{proof}
The first-order critical points of $F$ are the first-order critical points of 
\[
\Bigl\langle \sT, \bigotimes_{i=1}^\ell \bigl(\fv^{(i)}\bigr)^{\otimes a_i} \Bigr\rangle,
\]
which are the pSVTs of $\sT$, by~\cite{lim2005singular}.
For $i \in [\ell ]$, let $F_i$ be the restriction of $F$ to the $i$-th block with the other blocks fixed. 
The gradient and Hessian are
\begin{align}
\nabla F_i (\fv^{(i)})
&= a_i \,\sT\cdot \bigotimes_{j=1}^\ell \bigl(\fv^{(j)}\bigr)^{\otimes (a_j - \delta_{i,j})}
+ a_i \gamma_i \|\fv^{(i)}\|^{a_i-2} \fv^{(i)}, \label{eq:grad_Fi} \\
\nabla^2 F_i (\fv^{(i)})
&= a_i(a_i-1) \,\sT\cdot \bigotimes_{j=1}^\ell \bigl(\fv^{(j)}\bigr)^{\otimes (a_j - 2\delta_{i,j})}
+ a_i(a_i-2)\gamma_i\|\fv^{(i)}\|^{a_i-4}(\fv^{(i)})^{\otimes 2} + a_i\gamma_i  \|\fv^{(i)}\|^{a_i-2}I. \label{eq:hess_Fi}
\end{align}
For convexity of $F_i$, we require $\fy\T \nabla^2 F_i(\fv^{(i)})\fy \geq0$ for all $\fy\in \SB^{m_i-1}$.  
If $a_i \geq 2$, this is bounded from below by
$-a_i(a_i-1)\|\sT\|_2 + a_i\gamma_i$,
so it is non-negative if $\gamma_i \geq(a_i-1)\|\sT\|_2$, 
using \eqref{eq:hess_Fi} and $\|\fv^{(i)}\|=1$. If $a_i=1$, the inequality holds by Cauchy-Schwarz.  

We show that each block update in Algorithm \ref{alg:psspm} increases the value of $F$.
The update for $\fv^{(i)}$ is 
\[
\fv^{(i)\prime} = \frac{\sT\cdot \bigotimes_{j=1}^\ell \bigl(\fv^{(j)}\bigr)^{\otimes a_j -\delta_{i,j}}
				      \;+\; \gamma_i \fv^{(i)}}{\| \sT\cdot \bigotimes_{j=1}^\ell \bigl(\fv^{(j)}\bigr)^{\otimes a_j -\delta_{i,j}}
				      \;+\; \gamma_i \fv^{(i)}\|} = \frac{\nabla F_i(\fv^{(i)})}{\|\nabla F_i(\fv^{(i)})\|}, \quad \text{for } \nabla F_i(\fv^{(i)}) \neq 0.
\]
If $\nabla F_i(\fv^{(i)})=0$, we set $\fv^{(i)\prime} = \fv^{(i)}$.
When $F_i$ is convex, updating a unit vector $\fv^{(i)}$ by normalizing the gradient of $F_i$ ensures that $F$ increases, unless $\fv^{(i)}=\fv^{(i)\prime}$, see e.g. \cite{kofidis2002best,regalia2003monotonic} for a proof and~\cite{kolda2011shifted} for a tensor application. 
We quantify the increase. Let 
$\fv=(\fv^{(1)},\ldots,\fv^{(\ell)})$ and $\fv'=(\fv^{(1)},\ldots,\fv^{(i)\prime},\ldots,\fv^{(\ell)})$.  
Convexity of $F_i$ gives
\begin{align}\label{eq:convexity}
F(\fv')-F(\fv)
&\geq\langle \nabla F_i(\fv),\, \fv^{(i)\prime}-\fv^{(i)} \rangle \\&= \|\nabla F_i(\fv)\| \langle \fv^{(i)\prime}, \fv^{(i)\prime}-\fv^{(i)}\rangle \\
&= \tfrac{1}{2}\|\nabla F_i(\fv)\| \,\|\fv^{(i)\prime}-\fv^{(i)}\|^2.
\end{align}
The Riemannian gradient of $F_i$ on the sphere $\SB^{m_i-1}$ is
\( \nabla_{\SB^{m_i-1}} F_i(\fv)
= (I-\fv^{(i)}\fv^{(i)\T})\nabla F_i(\fv)\).
Substituting $\nabla F_i = \|\nabla F_i\|\fv^{(i)\prime}$ yields
\begin{equation}\label{eq: riemannain euclidean gradient and distance}
\|\nabla_{\SB^{m_i-1}} F_i(\fv)\|^2
= \|\nabla F_i(\fv)\|^2 \bigl(1-\langle \fv^{(i)},\fv^{(i)\prime}\rangle^2\bigr)
\;\le\;\|\nabla F_i(\fv)\|^2 \,\|\fv^{(i)\prime}-\fv^{(i)}\|^2.
\end{equation}
Combining \eqref{eq:convexity} and 
\eqref{eq: riemannain euclidean gradient and distance}
gives
\begin{equation}\label{eq:function_increase}
F(\fv')-F(\fv)\;\ge\;\tfrac{1}{2}\|\nabla_{\SB^{m_i-1}} F_i(\fv)\| \,\|\fv^{(i)\prime}-\fv^{(i)}\|.
\end{equation}

We now prove the convergence of Algorithm~\ref{alg:psspm}. 
The proof adapts \cite[Theorem~2.3]{schneider2015convergence} to blockwise updates.  
Let
\[
\fv_{(k,i)} = \bigl(\fv^{(1,k+1)}, \ldots, \fv^{(i-1,k+1)}, \fv^{(i,k+1)}, \fv^{(i+1,k)}, \ldots, \fv^{(\ell,k)}\bigr)
\]
be the tuple in the $k$-th iteration after updating block $i$ and set $\fv_{(k,\ell+1)} = \fv_{(k+1,1)}$. 
The sequence $\{\fv_{(k,i)}\}_{k \geq0, 1\leq i\leq \ell}$ has a subsequence that converges to a point
\[
\fv_\ast = \bigl(\fv^{(1)}_\ast,\ldots,\fv^{(\ell)}_\ast\bigr),
\]
since the ambient space $\mathbb{S}^{m_1 - 1} \times \cdots \times \mathbb{S}^{m_\ell - 1}$ is compact. 
Each $F_i$ is real-analytic on $\mathbb{S}^{m_i - 1}$. 
Hence there exist constants $\theta_i \in [\tfrac{1}{2}, 1)$, $\Lambda_i > 0$, and $\sigma > 0$ such that for all $\fv^{(i)}$ with $\|\fv^{(i)}-\fv^{(i)}_\ast\|\leq\sigma$, the Łojasiewicz inequality \cite[p92]{law1965ensembles} holds: 
\begin{equation}\label{eq: lojasiewicz}
\bigl|F_i(\fv^{(i)})-F_i(\fv^{(i)}_\ast)\bigr|^{\theta_i}
\;\le\;\Lambda_i \bigl\|\nabla_{\SB^{m_i-1}}F_i(\fv^{(i)})\bigr\|.
\end{equation}
Set $\theta=\max_i \theta_i$ and $\Lambda=\max_i \Lambda_i$. 
The sequence $\{F(\fv_{(k,i)})\}_{k\geq0}$ increases monotonically, so it converges to $F(\fv_\ast)$.  
Suppose $\fv_{(k,i)}$ is close enough to $\fv_\ast$ and $k$ is large so that
\begin{equation}\label{eq:large k}
\|\fv_{(k,i)}-\fv_\ast\| \leq\tfrac{\sigma}{2},
\quad
\bigl(F(\fv_\ast)-F(\fv_{(k,i)}) \bigr)^{1-\theta} \leq \min\{\tfrac{\sigma}{4\Lambda},\frac{1}{2}\}.
\end{equation}
Then equation \eqref{eq: lojasiewicz} holds for $\fv_{(k,i)}$ and together with \eqref{eq:function_increase}, we obtain 
\[
\bigl(F(\fv_\ast)-F(\fv_{(k,i)})\bigr)^{\theta_i}{\|\fv^{(i,k+1)}-\fv^{(i,k)}\|}
\;\le\;2\Lambda_i (F(\fv_{(k,i+1)})-F(\fv_{(k,i)})).
\]
We define $G(\fv):=F(\fv_\ast)-F(\fv)$, whose function value is non-increasing after each update. 
If $G(\fv_{(k,i)}) = 0$, then $\fv_\ast = \fv_{k,i}$ and the proof is complete.
Otherwise, we have
\begin{align}
\|\fv^{(i,k+1)}-\fv^{(i,k)}\| &\;\le\;2\Lambda_i\frac{G(\fv_{(k,i)})-G(\fv_{(k,i+1)})}{G(\fv_{(k,i)})^{\theta_i}}\\
&\;\le\;2\Lambda\bigl(G(\fv_{(k,i)})^{1-\theta_i}-G(\fv_{(k,i+1)})^{1-\theta_i}\bigr)\\
\label{eq: difference to sum}
&\;\le\;2\Lambda\bigl(G(\fv_{(k,i)})^{1-\theta}-G(\fv_{(k,i+1)})^{1-\theta}\bigr)\\
& \;\le\; 2\Lambda G(\fv_{(k,i)})^{1-\theta} \leq\tfrac{\sigma}{2},
\end{align}
 where the second last inequality follows from the fact that the function $f(x) = x^{1-\theta}- x^{1-\theta_i}$ takes non-negative values and is monotonically increasing for small $x>0$ and the last inequality follows from~\eqref{eq:large k}.
By the triangle inequality, \( \fv_{(k,i+1)} \) also satisfies $\|\fv_{(k,i+1)}- \fv_\ast\|\leq \sigma$, so the Łojasiewicz inequality holds for $\fv_{(k,i+1)}$.
Hence all iterates \( \fv_{(k',j)} \) after \( \fv_{(k,i)} \) satisfy
$$
\|\fv_{(k',j)} - \fv_\ast\| \leq\frac{\sigma}{2},
$$
by induction, 
so the Łojasiewicz inequality holds for all \( \fv_{(k',j)} \) around \( \fv_\ast \).
Summing over iterates, we obtain from \eqref{eq: difference to sum} that 
\begin{align}\label{eq:sum-small}
\sum_{(k,i)<(m,n)\leq(k',j)} \|\fv_{(m,n)}-\fv_{(m,n-1)}\|
&\leq\sum_{(k,i)<(m,n)\leq(k',j)} 2\Lambda (G(\fv_{(m,n-1)})^{1-\theta}-G(\fv_{(m,n)})^{1-\theta}) \\
&\le
2\Lambda\,G(\fv_{(k,i)})^{1-\theta}
<\tfrac{\sigma}{2}.
\end{align}
Therefore, $\{\fv_{(k,i)}\}$ is Cauchy and converges to $\fv_\ast \in \SB^{m_1-1}\times\cdots\times \SB^{m_\ell-1}$. The convergence rate is at least algebraic by \cite[Theorem 2.3]{schneider2015convergence}.

We are left to check that \( \fv_\ast \) is a pSVT of $\sT$.
Taking limits on the update rule of Algorithm~\ref{alg:psspm}, we have
$$
\fv^{(i)}_\ast \| \sT\cdot \bigotimes_{j=1}^\ell (\fv_\ast^{(j)})^{\otimes (a_i-\delta_{i,j})} + \gamma_i \fv_\ast^{(i)} \| =  \sT\cdot \bigotimes_{j=1}^\ell (\fv_\ast^{(j)})^{\otimes (a_i-\delta_{i,j})} + \gamma_i \fv_\ast^{(i)},
$$ 
which implies 
$
\sigma\fv^{(i)}_\ast  = \sT\cdot \bigotimes_{j=1}^\ell ({\fv_\ast^{(j)}})^{\otimes (a_i-\delta_{i,j})},
$
where $\sigma = \sT \cdot \bigotimes_{j=1}^\ell (\fv_\ast^{(j)})^{\otimes a_i}$.
\end{proof}

\begin{remark}
In the above proof, we update the blocks in a fixed order in each outer iteration.
Global convergence still holds for any order with every block updated at least once per outer iteration.
\end{remark}

\subsection{Local convergence}

In the following result, we show local linear convergence of Algorithm~\ref{alg:psspm} to certain pSVTs of a tensor with two blocks (i.e. $\ell = 2$). 
This generalizes~\cite[Theorem 4.10]{kileel2025subspace}, which is the case $\ell = 1$.

\begin{proposition}\label{prop:local linear}
Let 
\(
\sT \in S^{d_1}(\RR^{m_1}) \otimes S^{d_2}(\RR^{m_2})
\)
be a partially symmetric tensor whose pSVTs have singular value at most one.
Let $\gamma_1,\gamma_2 \geq 0$ be such that
\[
F(\fv^{(1)},\fv^{(2)}) 
= \langle \sT, (\fv^{(1)})^{\otimes d_1} \otimes (\fv^{(2)})^{\otimes d_2} \rangle 
+ \gamma_1 \|\fv^{(1)}\|^{d_1} + \gamma_2 \|\fv^{(2)}\|^{d_2}
\]
is convex in each block. 
Then, critical points of $F$ are pSVTs of $\sT$. 
PS–HOPM on $\sT$ with shifts $\gamma_1,\gamma_2$ has local linear convergence around every pSVT that is a second-order critical point of $F$ with positive singular value on the product of spheres ${\SB^{m_1-1}\times \SB^{m_2-1}}$.
\end{proposition}

\begin{proof}
Let $F_1,F_2$ denote the restrictions of $F$ to the first and second blocks, respectively.  They are convex, by assumption. 
Denote the Euclidean gradient of $F_i$ over the $i$-th block by $\nabla F_i(\fv^{(i)})$.
The PS–HOPM update in block $1$ is the normalized gradient map
\[
\Psi_1(\fv^{(1)}, \fv^{(2)}) = \left(\frac{\nabla F_1(\fv^{(1)})}{\|\nabla F_1(\fv^{(1)})\|}, \fv^{(2)}\right),
\]
and similarly for block $2$ and $\Psi_2$.
The overall update is the composition 
\(
\Psi = \Psi_2 \circ \Psi_1.
\)
A 
second-order critical point  
$\fv_\star=(\fv^{(1)},\fv^{(2)})$ with singular value $\sigma$ of $F$ is a fixed point of this map.  
Local linear convergence follows if the Jacobian 
\[
J\Psi = J\Psi_2(\fv_\star)\, J\Psi_1(\fv_\star)
\]
has spectral radius strictly less than one, by the center–stable manifold theorem (see, e.g., \cite[Theorem~3.5]{rheinboldt1998methods}).   

We compute the Jacobians $J\Psi_2, J\Psi_1$ and relate them to the Riemannian Hessian of $F$ at $\fv_\star$.  
Let $\mathbf{H}_{i,j}$ denote the $(i,j)$–block of the Riemannian Hessian of $F$, where $i,j\in\{1,2\}$ and define $\mathbf{P}_i = \Id - \fv^{(i)}\fv^{(i)\T}$ and $\alpha_i =\|\nabla F_i(\fv^{(i)})\|$. 
The gradient of $F_i$ is
$\nabla F_i (\fv^{(i)})
= d_i \,\sT\cdot \bigotimes_{j=1}^2 \bigl(\fv^{(j)}\bigr)^{\otimes (d_j - \delta_{i,j})}
+ d_i \gamma_i \fv^{(i)} = d_i(\sigma+\gamma_i) \fv^{(i)}$, so
 $\alpha_i = d_i(\sigma+ \gamma_i)>0$ by the assumption $\sigma>0$. 
We have 
\begin{align*}
\frac{\partial \Psi_i(\fv^{(i)})}{\partial \fv^{(i)}_j} =  \frac{\tfrac{\partial}{\partial \fv^{(i)}_j} \nabla F_i(\fv^{(i)})}{\|\nabla F_i(\fv^{(i)})\|}
-
\frac{\nabla F_i(\fv^{(i)})}{\|\nabla F_i(\fv^{(i)})\|^2} 
 \frac{\bigl\langle \nabla F_i(\fv^{(i)}), \;\tfrac{\partial}{\partial \fv^{(i)}_j} \nabla F_i(\fv^{(i)}) \bigr\rangle}{\|\nabla F_i(\fv^{(i)})\|}.
\end{align*}
It follows that the $(1,1)$–block of the Jacobian $J\Psi_1(\fv^{(1)})$ (as a linear operator on the tangent space $T_{\fv^{(1)}}\SB^{m_1-1}$) is
\begin{equation}\label{eq: J11}
(J\Psi_1)_{1,1} =  \mathbf{P}_1 \Biggl( 
\frac{\nabla^2 F_1(\fv^{(1)})}{\alpha_1} 
- \frac{\nabla^2 F_1(\fv^{(1)}) \nabla F_1(\fv^{(1)}) \nabla F_1(\fv^{(1)})\T}{\alpha_1^3}
\Biggr)\mathbf{P}_1 = \mathbf{P}_1 \frac{\nabla^2 F_1(\fv^{(1)})}{\alpha_1} \mathbf{P}_1,
\end{equation}
where the second term is zero
since $\nabla F_1(\fv^{(1)}) = \alpha_1 \fv^{(1)}$ at the fixed point $\fv_\star$.  
The $(1,1)$ Riemannian Hessian block is 
\begin{equation}\label{eq: H1,1}
\mathbf{H}_{1,1} = \mathbf{P}_1\bigl(\nabla^2 F_1(\fv^{(1)}) - \langle \nabla F_1(\fv^{(1)}), \fv^{(1)}\rangle \Id\bigr)\mathbf{P}_1 = \alpha_1 (J\Psi_1)_{1,1} - \alpha_1 \mathbf{P}_1,
\end{equation}
so we have
\begin{equation}\label{eq: J11 H11}
(J\Psi_1)_{1,1}
= \frac{\mathbf{H}_{1,1}}{\alpha_1} +\mathbf{P}_1.
\end{equation}
Similarly, the $(1,2)$–block of the Jacobian $J\Psi_{1}$ is 
\[
(J\Psi_1)_{1,2}
= \frac{\mathbf{H}_{1,2}}{\alpha_1}.
\]
Altogether, the Jacobian of the full iteration is 
\[
J \Psi= 
\begin{pmatrix}
\Id & 0 \\
\frac{\mathbf{H}_{2,1}}{\alpha_2} & \frac{\mathbf{H}_{2,2}}{\alpha_2} + \mathbf{P}_2
\end{pmatrix}
\begin{pmatrix}
\frac{\mathbf{H}_{1,1}}{\alpha_1}+\mathbf{P}_1 & \frac{\mathbf{H}_{1,2}}{\alpha_1} \\
0 & \Id
\end{pmatrix}.
\]
The matrix $(J\Psi_1)_{1,1}$ is symmetric and its eigenvalues are at least $0$  by \eqref{eq: J11} and $F_1$ being convex, and they are at most $1$ by \eqref{eq: J11 H11} and since $\mathbf{H}_{1,1}$ is negative definite, being a submatrix of the Hessian at a maximizer.

To analyze the spectrum, suppose $(\fv,\fw)$ is an normalized eigenvector of $J\Psi$ with eigenvalue $\lambda$.  
We will show that $\lambda$ is real.
Expanding $J\Psi (\fv,\fw ) =\lambda (\fv, \fw)$ gives
\begin{equation}\label{eq:vw}
\lambda \fv = \left(\frac{\mathbf{H}_{1,1}}{\alpha_1} + \mathbf{P}_1\right)\fv + \frac{\mathbf{H}_{1,2}}{\alpha_1}\fw,
\qquad
\lambda \fw = \frac{\mathbf{H}_{2,1}}{\alpha_2} \lambda\fv + \left(\frac{\mathbf{H}_{2,2}}{\alpha_2}+\mathbf{P}_2\right)\fw.
\end{equation}
Multiplying the equations \eqref{eq:vw} with the Hermitian transposes $\fv^H$ and $\fw^H$ respectively, and setting
\begin{equation}\label{eq:beta12mu}
\beta_1 = \fv^H \Bigl(\frac{\mathbf{H}_{1,1}}{\alpha_1}+\mathbf{P}_1\Bigr)\fv, \quad
\beta_2 = \fw^H \Bigl(\frac{\mathbf{H}_{2,2}}{\alpha_2}+\mathbf{P}_2\Bigr)\fw, \quad
\mu = \fv^H \mathbf{H}_{1,2}\fw ,
\end{equation}
we obtain
\begin{equation}\label{eq: lambda mu}
\lambda = \beta_1 + \frac{\mu}{\alpha_1}, 
\qquad
\lambda = \frac{\lambda}{\alpha_2}\bar{\mu} + \beta_2.
\end{equation}
Note that $\beta_1,\beta_2$ are real non-negative numbers with $\beta_1 + \beta_2\leq 1$ since the symmetric matrices $\frac{\mathbf{H}_{1,1}}{\alpha_1} + \mathbf{P}_1, \frac{\mathbf{H}_{2,2}}{\alpha_2}+\mathbf{P}_2$ have eigenvalues in $[0,1]$ and $\|\fv\|^2 +  \|\fw \|^2 = 1$.
Multiplying the first equation in \eqref{eq: lambda mu} by $\alpha_1\bar{\lambda}$, the second equation by $\alpha_2$ and taking the sum, we have
$$
\alpha_1\|\lambda\|^2 +  \lambda \alpha_2 = \beta_1\alpha_1 \bar{\lambda}  + \alpha_2\beta_2+ \mu \bar{\lambda} + \bar{\mu}\lambda.
$$
It follows that $\beta_1\alpha_1\bar{\lambda} - \lambda \alpha_2$ is real. Then either $\lambda$ is real or $\alpha_2 + \beta_1\alpha_1 = 0$. The second case cannot happen since $\alpha_2, \alpha_1>0, \beta_1\geq 0$.
Therefore, $\lambda\in\RR$ and so is $\mu.$

We now show that $0 \leq \lambda < 1$.  
If $\mu < 0$, then $\lambda = \beta_1 + \tfrac{\mu}{\alpha_1} < \beta_1 \leq 1$.  
Assume $\lambda < 0$, the relation $\lambda = \beta_2 + \tfrac{\lambda}{\alpha_2}\mu$ implies $\tfrac{\lambda}{\alpha_2}\mu< 0$ and hence $\mu > 0$, a contradiction.  
Thus, $0 \leq \lambda < 1$. 
If $\mu = 0$, then $0 \leq \lambda = \beta_1 = \beta_2 \leq \frac{1}{2},$ since $\beta_1,\beta_2 \geq 0$ and $\beta_1 + \beta_2 = 1$.
If instead $\mu > 0$, then $\lambda = \beta_1 + \tfrac{\mu}{\alpha_1} > 0$. 
Then 
$$2\mu + \alpha_1\beta_1 + \alpha_2\beta_2 \leq \alpha_1 \|\fv\|^2 + \alpha_2 \|\fw\|^2 \leq \max\{\alpha_1,\alpha_2\},$$ by \eqref{eq:beta12mu} and the negative definiteness of the Hessian of $F$. 
Hence $\max\{\alpha_1,\alpha_2\} - \mu > 0$.
Multiplying the equations of \eqref{eq: lambda mu} by $\alpha_1$ and $\alpha_2$, respectively, and by the above inequality, we obtain 
\begin{align*}
(\alpha_1 + \alpha_2) \lambda = (\lambda-1) \mu + \alpha_1\beta_1 + \alpha_2 \beta_2 + 2\mu  < \max\{\alpha_1,\alpha_2\} + (\lambda-1) \mu.
\end{align*}
Hence,  
$\lambda < \frac{\max\{\alpha_1,\alpha_2\} - \mu}{\alpha_1 + \alpha_2 - \mu}<1.$
\end{proof}

We have shown local linear convergence to a global maximizer of $F_\Acal$. In practice,  PS-HOPM  may converge to a critical point of $F_\Acal$ that is not globally optimal.   
We show that a critical point of $F_\Acal$ has low objective value or is close to an optimal rank-one component.

\begin{proposition}
Fix
\(
\sT^{(f)} = \sum_{i=1}^r \lambda_i \bigotimes_{j=1}^\ell\,
(\fv^{(j)}_i)^{\otimes f_j}\otimes \fu_i,
\)
with last slices orthonormal with span $\Acal$. 
Define the rank-one tensors
\[
\sR^{(i)} := \bigotimes_{j=1}^\ell\,
(\fv^{(j)}_i)^{\otimes f_j}, 
\qquad i=1,\ldots,r.
\] 
Assume that $\Acal = \Span\{\sR^{(1)},\ldots, \sR^{(r)}\}.$
For $(\fw^{(1)},\ldots,\fw^{(\ell)})\in \SB^{m_1-1}\times\cdots\times \SB^{m_\ell-1}$, let 
$\sT_\fw := \bigotimes_{j=1}^\ell (\fw^{(j)})^{\otimes f_j}.$
Then the objective function $F_\Acal$ satisfies
\[
\max_i |\langle \sR^{(i)}, \sT_\fw \rangle|^2
\;\le\; F_\Acal(\fw^{(1)},\ldots,\fw^{(\ell)})
\;\le\;
C \max_i |\langle \sR^{(i)}, \sT_\fw \rangle|,
\]
where $C = {({1-\rho})^{-1}}\Bigl(\prod_{k=1}^\ell (1+(r-1)\rho_k)\Bigr)^{1/\ell}$, 
\(
\rho := \max_{i\ne j} |\langle \sR^{(i)}, \sR^{(j)} \rangle| \) and,  for \(k=1,\ldots,\ell\), 
\(
\rho_k := \max_{i\ne j} |\langle \fv^{(k)}_i, \fv^{(k)}_j \rangle|\).
\end{proposition}
\begin{proof}
We have
\[
F_\Acal(\fw^{(1)},\ldots,\fw^{(\ell)}) 
= \|P_\Acal(\sT_\fw)\|^2
\;\ge\; \max_i\|P_{\Span\{\sT^{(i)} \}}(\sT_\fw)\|^2
= \max_i |\langle \sR^{(i)}, \sT_\fw\rangle|^2.
\]
This gives the lower bound.
For the upper bound,
there exist scalars $\mu_1,\ldots,\mu_r$ such that
\(
P_\Acal(\sT_\fw) = \sum_{i=1}^r \mu_i \sR^{(i)},
\)
so that
\(
F_\Acal(\fw^{(1)},\ldots,\fw^{(\ell)}) = \sum_{i=1}^r \mu_i \langle \sR^{(i)}, \sT_\fw\rangle.
\)
Since~$P_\Acal(\sR^{(i)})= \sR^{(i)}$, we have
\[
\langle P_\Acal(\sT_\fw), \sR^{(i)}\rangle = \langle \sT_\fw, \sR^{(i)}\rangle
= \mu_i + \sum_{j\ne i} \mu_j \langle \sR^{(j)}, \sR^{(i)}\rangle.
\]
We obtain
\[
|\mu_i - \langle \sR^{(i)}, \sT_\fw\rangle|
\leq\rho \|\mu\|_\infty
\Rightarrow
|\mu_i|\leq\rho \|\mu\|_\infty + |\langle \sR^{(i)},\sT_\fw\rangle|,
\]
by comparing these two equalities.
This yields an upper bound 
\[
\|\mu\|_\infty \leq\frac{\max_i |\langle \sR^{(i)}, \sT_\fw\rangle|}{1-\rho}.
\]
Plugging this into the expression for $F_\Acal$, we obtain
\begin{equation}\label{eq:bound1}
F_\Acal(\fw^{(1)},\ldots,\fw^{(\ell)})
\leq\frac{\sum_{i=1}^r |\langle \sR^{(i)}, \sT_\fw\rangle|}{1-\rho} \max_i |\langle \sR^{(i)}, \sT_\fw\rangle|.
\end{equation}
We have
\begin{equation}\label{eq:holder}
\sum_{i=1}^r |\langle \sR^{(i)}, \sT_\fw\rangle|
= \sum_{i=1}^r \prod_{j=1}^\ell |\langle \fv^{(j)}_i, \fw^{(j)}\rangle|^{f_j}
\;\le\; \Biggl(\prod_{j=1}^\ell \sum_{i=1}^r |\langle \fv^{(j)}_i, \fw^{(j)}\rangle|^{\ell f_j}\Biggr)^{1/\ell},
\end{equation}
by Hölder's inequality.
Let $\mathbf{M}_j$ be the $r\times m_j$ matrix with rows $\fv^{(j)}_1,\ldots,\fv^{(j)}_r$. 
Then,
\begin{align}\label{eq:gershgorin}
\sum_{i=1}^r |\langle \fv^{(j)}_i, \fw^{(j)}\rangle|^{\ell f_j}
&\leq\sum_{i=1}^r |\langle \fv^{(j)}_i, \fw^{(j)}\rangle|^2
= \|\mathbf{M}_j \fw^{(j)}\|_2^2
\leq\| \mathbf{M}_j \mathbf{M}_j\T  \|_2
\leq1+(r-1)\rho_j,
\end{align}
where the last inequality follows from Gershgorin’s circle theorem~~\cite{horn2012matrix}.
Hence, we obtain the upper bound
\[
F_{\mathcal{A}}(\fw_1,\dots,\fw_\ell)
\leq\frac{\left(\prod_{j=1}^{\ell}(1+(r-1)\rho_j)\right)^{\frac{1}{\ell}}}{1-\rho}\max_i|\langle \sR^{(i)},\sT_\fw\rangle|,
\]
from \eqref{eq:bound1}, \eqref{eq:holder}, and \eqref{eq:gershgorin}.
\end{proof}

\begin{remark}
When all the rows of each $\bM_j$ are orthonormal, we have $\rho=\rho_1=\cdots=\rho_\ell=0$, and hence
\(
\max_i |\langle \sR^{(i)}, \sT_\fw\rangle|^2
\;\le\; F_\Acal(\fw^{(1)},\ldots,\fw^{(\ell)})
\;\le\; \max_i |\langle \sR^{(i)}, \sT_\fw\rangle|\).
\end{remark}

\section{The MSPM algorithm}
In this section, we show how to recover a rank-$r$ tensor decomposition via MSPM. The procedure has four steps. We focus on Steps (1), (3), and (4). Step (1) transforms the input tensor into an auxiliary tensor with orthonormal slices, where rank-one summands correspond to pSVTs with singular value one (Section~\ref{subsec: tensor transformation}). Step (2), which computes these pSVTs via a power method, was developed in Section~\ref{sec: power method}. Step (3) completes the partal factors from the pSVTs to full rank-one summands (Section~\ref{subsec:completion}) then extracts all components via deflation (Section~\ref{subsec:deflation}). Step (4) repeats the first three steps until full decomposition
(Section~\ref{subsec:practical_considerations}).

\subsection{Transforming a tensor to one with orthonormal slices}\label{subsec: tensor transformation}

In Section~\ref{sec:decomposition via svt}, we studied tensors with orthonormal slices and established a correspondence between the summands in their decomposition and the pSVTs with singular value one.
For $\sT\in \bigotimes_{i=1}^\ell S^{d_i}(\RR^{m_i})$, we build a tensor with orthonormal slices, as follows (see Figure \ref{fig:flattening_pipeline}).
\begin{enumerate}
    \item Given any $f = (f_1, \ldots, f_\ell)$ with $0 \leq f_i \leq d_i$ for all $i$, build the flattening $\mathbf{T}^{(f)}$ of $\sT$.
    \item Compute an orthonormal basis for its column space (by SVD or QR decomposition).
    \item Reshape each basis vector into a partially symmetric tensor in $ \bigotimes_{i=1}^\ell S^{f_i}(\RR^{m_i})$. 
    \item Stack these to form the last slices of tensor~$\sT^{(f)}$ with symmetry $(f_1,\ldots,f_\ell,1)$.
\end{enumerate}

When the input tensor is noisy, one may truncate the SVD in Step~2 by retaining the top $r$ singular components, to a desired rank, or by applying a threshold to the singular values.

If the tensor $\sT$ has rank $r$ with decomposition
\[
\sT
=\sum_{i=1}^r \lambda_i\, \bigotimes_{j=1}^\ell(\fv_{i}^{(j)})^{\otimes d_j},\]
 then the associated tensor $\sT^{(f)}$ can be written as
\[
\sT^{(f)}
= \sum_{i=1}^r \mu_i\,
\left( \bigotimes_{j=1}^\ell (\fv_i^{(j)})^{\otimes f_j} \right)
\otimes
\fu_i,
\]
for some $\mu_i\in \RR, \fu_i\in \RR^{\prod_{i=1}^{\ell} m_i^{d_i-f_i}}$.
If $\sT^{(f)}$ has unique rank $r$ decomposition, we can recover the factors ${\fv_i^{(j)}}$ with $f_j\neq 0$ in $\sT$ from the decomposition of $\sT^{(f)}$, thereby reconstructing part of the decomposition of $\sT$.
If the rank of $\sT^{(f)}$ equals
$\rank \bT^{(f)}$, we can use the pSVTs of $\sT^{(f)}$ with singular value one to decompose $\sT^{(f)}$ by Theorem \ref{thm:decomp and singular vector}.

The condition that $\sT^{(f)}$ has unique rank $r = \rank \bT^{(f)}$ decomposition
holds 
for generic rank-one summands of $\sT$
if $r$ is smaller than a quantity related to $f$, which we now define.

\begin{definition}[Maximal rank $r(f)$]
\label{def:rmax}
Let $r(f)$ be the largest integer $r$ such that, for a generic partially symmetric tensor $\sT$ of rank $r$, the tensor $\sT^{(f)}$ and its flattening $\bT^{(f)}$ satisfy $$\rank(\sT^{(f)}) = \rank(\bT^{(f)})=r$$ and $\sT^{(f)}$ admit a unique rank-$r$ decomposition.
\end{definition}

The notion $r(f)$ is well-defined, as the uniqueness of the decomposition of $\sT^{(f)}$ does not depend on the orthonormal basis.
Suppose $\sT \in \bigotimes_{i=1}^\ell S^{d_i}(\mathbb{R}^{m_i}).$
The rank $r (f)$ cannot exceed the number of rows and columns of the flattening matrix $\bT^{(f)}$, which we denote by
\begin{equation}
    n_{\rm row}^f := \prod_{i=1}^{\ell} \binom{m_i + f_i - 1}{f_i} 
    = \dim\left( \bigotimes_{i=1}^\ell S^{f_i}(\mathbb{R}^{m_i}) \right), 
    \qquad 
    n_{\rm col}^f := \prod_{i=1}^{\ell} \binom{m_i + d_i - f_i - 1}{d_i - f_i}.
    \label{eqn:ncol}
\end{equation}
Here $n_{\rm row}^f$ and $n_{\rm col}^f$ count the number of distinct rows and columns, after removing repetitions due to partial symmetry.  
Another relevant quantity is the codimension of the rank-one tensors in $\bigotimes_{i=1}^\ell S^{f_i}(\mathbb{R}^{m_i})$, given by
\begin{equation}
    \label{eqn:ntrisecant}
    n_{\rm codim}^f := \prod_{i=1}^{\ell} \binom{m_i + f_i - 1}{f_i} - 1 - \sum_{f_i \neq 0} (m_i-1).
\end{equation}

\begin{proposition}
\label{thm:rmax_formula}
Let $\sT \in \bigotimes_{i=1}^\ell S^{d_i}(\mathbb{R}^{m_i})$ be a generic rank $r$ tensor, 
with flattening $\mathbf{T}^{(f)}$ from tuple 
$f = (f_1,\ldots,f_\ell)$. 
Form the tensor $\sT^{(f)}$ by stacking an orthonormal basis of $\mathbf{T}^{(f)}$, and define $n_{\rm row}^f,n_{\rm col}^f,n_{\rm codim}^f$ as in~\eqref{eqn:ncol} and \eqref{eqn:ntrisecant}.
Then the largest $r$ for which $\sT^{(f)}$ admits a 
unique rank $r = \rank \bT^{(f)}= \rank \sT$ decomposition of $\sT$ is
\begin{equation}\label{eq: r(f)}
r(f) =
\begin{cases}
\min\{\,n_{\rm col}^f,\, n_{\rm codim}^f + 1\,\}, & 
\text{if $f$ has two entries $i, j$ equal to $1$ (all others $0$)} \\[-0.5ex]
& \quad \text{with $\min\{m_i,m_j\}=2$,}\\
\min\{\,n_{\rm col}^f,\, n_{\rm codim}^f\,\}, & \text{otherwise}.
\end{cases}
\end{equation}
\end{proposition}
\begin{proof}
See Appendix \ref{app: proof of rmax}.
\end{proof}

The optimal flattening that maximizes $r(f)$ depends nontrivially on the tensor dimensions.
Figure~\ref{fig:rmax_plot} illustrates this behavior for $S^4(\mathbb{R}^{20}) \otimes \mathbb{R}^m$ as $m$ varies from 2 to 120, where we compare $r(f)$ across all possible flattenings.
In practice, selecting the optimal flattening requires evaluating $r(f)$ for each candidate. However, each $r(f)$ admits a simple closed-form expression, so this search is computationally negligible.

\begin{figure}[ht]
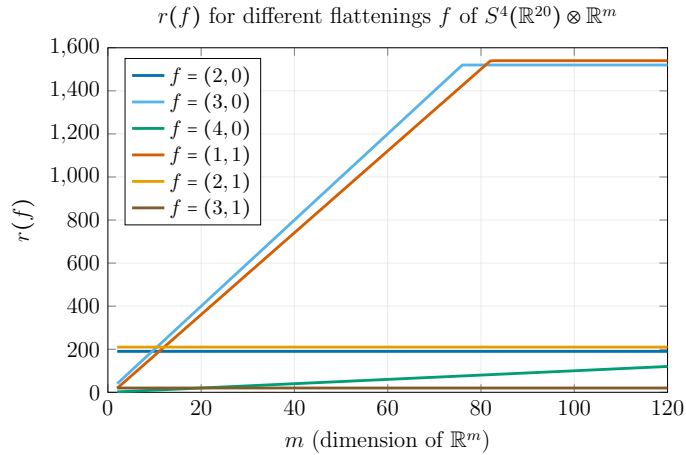

    \centering
    \includestandalone[width=0.6\textwidth]{tikzfigures_for_main/rmax_correct} 
     \caption{
        The maximal rank $r(f)$ for $S^4(\mathbb{R}^{20}) \otimes \mathbb{R}^m$ with 
        $m \in [2,120]$ for all flattenings $f$. The quantity $r(f)$ is piecewise linear in $m$ in this case.
    }
    \label{fig:rmax_plot}
\end{figure}

\subsection{Completing the decomposition}\label{subsec:completion}

We have seen that the pSVTs of $\sT^{(f)}$ allow us to recover partial rank-one components of $\sT$. 
Specifically, let
\[
\sT
= \sum_{i=1}^r \lambda_i\, \bigotimes_{j=1}^\ell (\fv_i^{(j)})^{\otimes d_j}.
\]
If $r \le r(f)$ (see Definition~\ref{def:rmax} and~\eqref{eq: r(f)}), then the partial component
\[
\sR_{(f)} := \bigotimes_{i=1}^\ell (\fv_1^{(i)})^{\otimes f_i}
\]
is recovered from the pSVTs of $\sT^{(f)}$ with singular value one.

We now describe how to recover its complement
\[ \sR_{(d-f)} = 
\bigotimes_{i=1}^\ell(\fv_{1}^{(i)})^{\otimes (d_i-f_i)}.
\]
The complement is determined immediately if $f_i > 0$ for all $i$. It remains to consider when $f_i = 0$ for some $i$. 
This is unavoidable when decomposing fully asymmetric tensors, and in the partially symmetric case can be advantageous, as it may lead to a flattening that can decompose tensors of higher rank.

\begin{definition}
Let \(d=(d_1,\ldots,d_\ell)\) and \(f=(f_1,\ldots,f_\ell)\) with \(d_i>0\) and \(0\leq f_i\leq d_i\).
The tuple \(f\) is \emph{symmetry breaking} for \(d\) if there exists an index \(i\) with \(f_i>0\) and \(d_i-f_i>0\).  
\end{definition}

We call such a tuple symmetry breaking becauase the same vector appears in both the tensor and its complement: after recovering \(\sR_{(f)} \) we know some of the vectors in \(\sR_{(d-f)}\). The tensor \(\sR_{(d-f)}\) can be reconstructed from those known factors together with the constraints of the row space
\(\mathbf{T}^{(f)}\), 
see Proposition~\ref{thm: match up patterns} and Figure~\ref{fig:symmetrybreaking}.

\begin{figure}[htbp]
  \centering
  \includegraphics[width=0.4\linewidth]{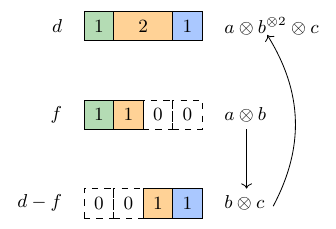}
  \caption{Symmetry–breaking case \(d=(1,2,1)\), \(f=(1,1,0)\): use shared vectors
  and the row space of \(\mathbf{T}^{(f)}\) to recover the full summand.}
  \label{fig:symmetrybreaking}
\end{figure}

When \(f\) is not symmetry breaking (for example \(d=(1,1,1,1)\),
\(f=(1,1,0,0)\)), one flattening is insufficient. In this case we use two tuples \(f,f'\) satisfying
\begin{equation}\label{eq:requirements on f1f2}
f' \neq d-f, \qquad \min(f,f')\neq 0, \qquad f \ngeq f',
\end{equation}
where $\min(f,f') \neq 0$ means $\exists i$ such that $\min(f_i,f'_i) \neq 0$ and $f \ngeq f'$ means $\exists i$ such that $f_i < f'_i$
The completion procedure is as follows (see Figure~\ref{fig:nonsymmetry})
\begin{enumerate}
\item Recover a rank-one tensor in \(\colspan(\mathbf{T}^{(f)})\).
\item Use the shared factors of $f$ and $f'$ to recover a rank-one tensor in \(\colspan(\mathbf{T}^{(f')})\).
\item If \(\min(f,f')>0\), recovery is complete. Otherwise, use the shared factors between $d-f$ and $f'$ to find the remaining factors. 
\end{enumerate}

\begin{figure}[htbp]
    \centering
    \includegraphics[width=0.5\linewidth]{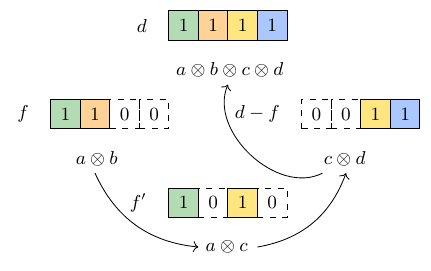}
    \caption{Steps to recover the rank-one tensor in the non-symmetry breaking case where $d = (1,1,1,1), f = (1,1,0,0), f' = (1,0,1,0)$.}
    \label{fig:nonsymmetry}
\end{figure}

We summarize the procedure for recovering $\sR_{(d-f)}$ in Algorithm \ref{alg: recover complement}.
Let $a = (a_1,\ldots,a_\ell)$ with $a_j = \min(f_j,d_j-f_j)$
We show how to recover $\sR_{(d-f)}$ using its shared factors with $\sR_{(f)}$.

\begin{algorithm}[htbp]
\caption{Recovery of complementary factors based on the type of flattening $f$}
\label{alg: recover complement}
\begin{algorithmic}[H]
\Require Flattening tuple $f$ and partially symmetric tensor $\sT 
=\sum_{i=1}^r \lambda_i \bigotimes_{j=1}^\ell (\fv_{i}^{(j)})^{\otimes d_j}$. Rank-one tensor $\sR_{(f)} = \bigotimes_{i=1}^\ell (\fv_{1}^{(i)})^{\otimes f_i}$. 
\Ensure Its complement $\sR_{(d-f)} = \bigotimes_{i=1}^\ell (\fv_{1}^{(i)})^{\otimes (d_i-f_i)}$.
\If{$f_i > 0$ for all $i$}
    \State All factors in $\sR_{(d-f)}$ also appear in $\sR_{(f)}$, so it can be found directly.
\ElsIf{$f$ is \textbf{symmetry breaking}}
    \State Some factors of $\sR_{(d-f)}$ are known.
    \State Reconstruct the remaining factors using the row span of $\mathbf{T}^{(f)}$ (see Figure~\ref{fig:symmetrybreaking}).
\Else 
    \State Choose tuple $f'$ such that
    \[
    f \neq d-f', \qquad \min(f,f')\neq 0, \qquad f \ngeq f'.
    \]
    \State Use the factors shared by $f$ and $f'$ to find a rank-one tensor in $\colspan(\mathbf{T}^{(f')})$.
    \If{$\min(f,f')>0$}
        \State Recovery is complete.
    \Else
        \State Use the factors shared by $f'$ and $d-f$ to recover the remaining factors from $\mathbf{T}^{(d-f)}$.
    \EndIf 
    \EndIf
\end{algorithmic}
\end{algorithm}

\begin{proposition}\label{thm: match up patterns}
Let $\Acal=\Span\Bigl\{\,\bigotimes_{j=1}^\ell (\fv_{i}^{(j)})^{\otimes (d_j-f_j)}\ :\ 1\leq i\leq r\,\Bigr\}$, where the $\fv_{i}^{(j)}$ are generic with norm one, and assume $\dim \Acal = r < r(d-f)$ with $r(d-f)$ as defined in \eqref{eq: r(f)}. 
Let $\sT^{(d-f)} \in \bigotimes_{j=1}^\ell S^{d_j - f_j}(\mathbb{R}^{m_j}) \otimes \mathbb{R}^r$
be the tensor formed by stacking an orthonormal basis of~$\Acal$.

For $1 \le i \le r$ and an integer tuple $a = (a_1,\ldots,a_\ell)$ with $0 < a_j < d_j - f_j$, define
\[
\sR^{(i)}_{(a)} := \bigotimes_{j=1}^\ell (\fv_i^{(j)})^{\otimes a_j}.
\]
Consider the contraction $\sT^{(d-f)} \cdot \sR^{(i)}_{(a)}$, and flatten to a matrix $\mathbf{M}$ with $r$ columns.
Then \(\Vect(\sR^{(i)}_{(d-f-a)})\) is the unique left singular vector of 
\(\mathbf{M}\) with
singular value \(1\). 
\end{proposition}

\begin{proof}
Let \((\fu,\fv)\) be a singular pair of $\bM$ with singular value $\sigma$.  
Let $\bM_\Acal$ be the flattening of $\sT^{(d-f)}$ via tuple $(d_1-f_1,\ldots,d_\ell-f_\ell,0)$ with $r$ orthonormal columns.
By definition of $\bM$, we have
\[
\bM_\Acal\T\,\Vect \bigl(\fu\otimes \sR^{(i)}_{(a)}\bigr) \;=\; \sigma \fv,
\]
where we regard $\fu\otimes \sR^{(i)}_{(a)}$ as an element of $\bigotimes_{j=1}^\ell S^{d_j-f_j}(\RR^{m_j})$ by appropriately reordering the tensor factors.  
We have
\[
\sigma = \bigl\|\,\bM_\Acal\T\Vect(\fu\otimes \sR^{(i)}_{(a)})\,\bigr\|_2
\leq\|\fu\otimes \sR^{(i)}_{(a)}\|_2=1,
\]
with equality when \(\fu\otimes \sR^{(i)}_{(a)}\in\Acal\), 
since the last slices of \(\sT^{(d-f)}\) are orthonormal.
Conversely, we assume 
\(\fu\otimes \sR^{(i)}_{(a)}\in\Acal\) and 
define $\fv = \bM_\Acal\T\Vect(\fu\otimes \sR^{(i)}_{(a)})$.
It follows that 
$(\Vect(\fu\otimes \sR^{(i)}_{(a)}), \fv)$ is a pSVT of $\bM_\Acal$ with singular value one. The tuple $(\fu,\fv)$ is a pSVT of $\bM$ with singular value one, by definition of $\bM$.

Let $\Bcal$ be the span of the left singular vectors of $\bM$ with singular value one.
Suppose that $\Vect(\sT') \in \Bcal$ is another rank-one tensor.  
Then 
$\sT' \otimes \sR^{(i)}_{(a)} \in \Acal$.  
Moreover, the $r$ rank-one tensors $\{\sR^{(i)}_{(d-f)}\}_{i=1}^r$ are the only rank-one elements in $\Acal$ (up to scale), since $r \leq r(d-f)$.  
It remains to exclude the possibility that $\sR^{(i)}_{(a)}$ and $\sR^{(j)}_{(a)}$ are collinear for some $i \neq j$.  
This cannot occur under the genericity assumption on the tensors $\{\sR^{(i)}_{(d-f)}\}_{i=1}^r$, which ensures that their corresponding factors are pairwise linearly independent.
\end{proof}

\begin{corollary}\label{cor: r_max with matching}
The largest rank for which one can uniquely recover a partially symmetric decomposition of a generic tensor $\sT\in \bigotimes_{i=1}^\ell S^{d_i}(\RR^{m_i})$ via finding pSVTs and matching up factors is
\begin{enumerate}
\item with a flattening \(f\) with all $f_i > 0$: 
$$r_{\max}(f) = r(f);$$
\item with a symmetry–breaking flattening \(f\):
\[
r_{\max}(f) = \min\{\,r(f),\ r(d-f)\,\}
=\min\{\,n_{\rm codim}^f,\ n_{\rm codim}^{\,d-f}\,\};
\]
\item with a pair \(f,f'\) that are not symmetry breaking and satisfy~\eqref{eq:requirements on f1f2}:
\[
r_{\max}(f,f') = \begin{cases}
\min\{\,r(f),\ r(f')\,\} & \text{if $\min\{f,f'\}>0$}, \\[0.5em]
 \min\{\,r(f),\ r(f'),\ r(d-f),\,\}, & \text{elsewhere}.
\end{cases}
\]

\end{enumerate}
\end{corollary}
\begin{proof}
    For unique decomposition of $\sT$, the largest rank $r$ need to satisfy $r\leq r(g)$ for any flattening tuple $g$ used either to extract rank-one tensors via pSVTs as in Theorem \ref{thm:decomp and singular vector} or to recover complementary factors as in Algorithm \ref{alg: recover complement}. The formulae then follow directly from Proposition~\ref{thm:rmax_formula} and~\ref{thm: match up patterns}.
\end{proof}

\subsection{Deflation}\label{subsec:deflation}

We have shown how to recover one rank-one summand of $\sT$. We now address how to obtain the remaining components. In principle, each summand can be reached from a suitable initialization, but this may require many random restarts. Instead, we use deflation: once a rank-one summand is identified, we subtract it from $\sT$.

In the following example, we examine the basins of attraction of the pSVTs corresponding to different rank-one summands. 
Figure~\ref{fig:basin} shows that, for a rank-three tensor, the basins have comparable size. 
If \(\sT\) has rank \(r \le r(f)\), then \(\sT^{(f)}\) has \(r\) pSVTs with singular value one. 
Assuming basins of attraction of comparable size, the classical coupon collector problem~\cite[pp.~224]{feller1991introduction} implies that recovering all \(r\) components via random initialization requires on the order of
\[
r \left(1 + \frac{1}{2} + \cdots + \frac{1}{r}\right) \approx r \log r
\]
runs. This is one motivation for using deflation: it reduces the cost from \(\Ocal(r \log r)\) to \(\Ocal(r)\).
Another motivation is that recovering rank-one summands sequentially simplify the optimization landscape since the set of rank-one tensors is closed and the size of \(\sT^{(f)}\) decreases after each deflation step.

\begin{example}[Basins of attraction of the rank-one summands]
Let
\[
\sT=\sum_{i=1}^3 \fa_i\otimes \fb_i\otimes \fc_i\in\RR^{3\times 3\times 3},
\qquad
\|\fa_i\|=\|\fb_i\|=\|\fc_i\|=1.
\]
Fix \(f=(1,1,0)\). Let \(\sT^{(f)} \in\RR^{3\times 3\times 3}\) be the stack of an orthonormal basis of
\(\Acal=\colspan(\mathbf{T}^{(f)})\).
Each \((\fa_i,\fb_i)\) is part of a pSVT of
\(\sT^{(f)}\) with singular value \(1\), by Theorem~\ref{thm:decomp and singular vector}.

We run PS-HOPM (Algorithm~\ref{alg:psspm}) on \(\sT^{(f)}\) with
\(\fc \) initialized to \([1,x,y]\T/\|(1,x,y)\|\) and \(\fa,\fb\) to the top singular vector pair of
\(\sT^{(f)}(*,*,\fc)\). 
For \(10000\) pairs \((x,y)\in[-1,1]^2\), every run
converges to a pSVT containing some \((\fa_i,\fb_i)\). The basins of attraction are in Figure~\ref{fig:basin}. 

\begin{figure}
    \centering
    \includegraphics[width=0.5\linewidth]{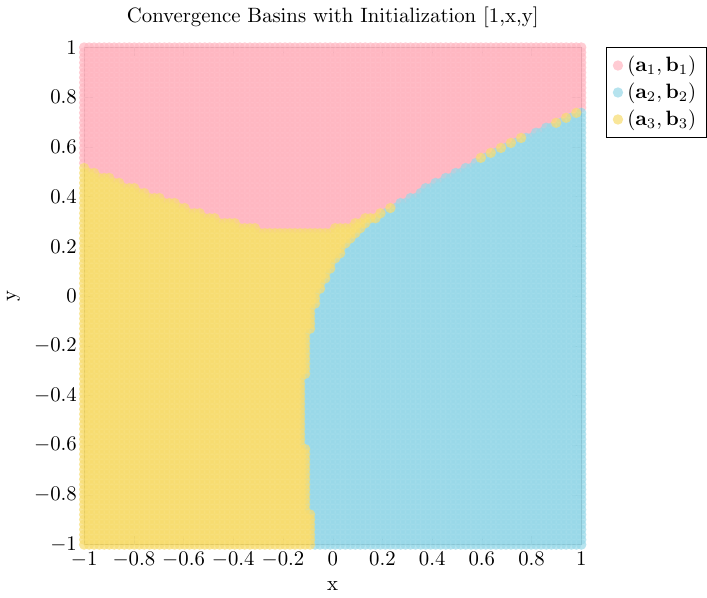}
    \caption{Basin of attraction for pSVTs of \(\sT^{(f)}\) under random initialization. 
Each point is an initialization vector \(\fc = [1, x, y]\T\). It is colored according to which rank-one component \((\fa_i, \fb_i)\) PS-HOPM converges to. 
The three pSVTs are all attractors, converged to with comparable frequency.
}
    \label{fig:basin}
\end{figure}
\end{example}

Let
\(
\sT
=\sum_{i=1}^r \lambda_i\, \bigotimes_{j=1}^\ell(\fv_{i}^{(j)})^{\otimes d_j}
\), 
where $\lambda_i\in\RR,\; \fv_{i}^{(j)}\in\SB^{m_j-1},$
and fix a tuple \(f=(f_1,\ldots,f_\ell)\).
Suppose we have identified a rank-one tensor 
\(
\sR_{(f)}:=\bigotimes_{i=1}^\ell(\fv_{1}^{(i)})^{\otimes f_i} \) and its complement \(
\sR_{(d-f)}:=\bigotimes_{i=1}^\ell(\fv_{1}^{(i)})^{\otimes (d_i-f_i)} \).
There is a unique \(\lambda_1\) such that the rank of
\[
\mathbf{T}^{(f)}
-\lambda_1\,\Vect(\sR_{(f)})\otimes\Vect(\sR_{(d-f)})
\]
drops by one,
by the Wedderburn rank reduction formula \cite{wedderburn1934lectures}. This \(\lambda_1\) is the coefficient of
\(\sR^{(1)}:=\bigotimes_{i=1}^\ell(\fv_{1}^{(i)})^{\otimes d_i}\)
in \(\sT\). Thus deflation removes \(\sT^{(1)}\) and reduces the problem
size by one.

The following proposition gives an explicit formula for $\lambda_1$ and orthonormal bases of the deflated subspaces.
The proof follows the same argument as \cite[Proposition~3.5]{kileel2025subspace}.

\begin{proposition}\label{prop:deflation}
Fix
\[
\sT
=\sum_{i=1}^r \lambda_i\, \bigotimes_{j=1}^\ell(\fv_{i}^{(j)})^{\otimes d_j},
\quad
f=(f_1,\ldots,f_\ell),\ 0\leq f_i\leq d_i.
\]
Let \(\mathbf{T}^{(f)}\) denote the \(f\)-flattening of \(\sT\) with
\[
\mathbf{T}^{(f)}=\mathbf{U} \mathbf{C}^{-1} \mathbf{V}\T,
\]
where \(\mathbf{U}\in\RR^{(\prod_{i=1}^\ell m_i^{\,f_i})\times r}, \mathbf{V}\in\RR^{(\prod_{i=1}^\ell m_i^{\,d_i-f_i})\times r}\) have orthonormal columns and 
\(\mathbf{C}\in\RR^{r\times r}\) is nonsingular.
Define \(\sR_{(f)}: = \bigotimes_{i=1}^\ell (\fv_{1}^{(i)})^{\otimes f_i}\) and \(\sR_{(d-f)}:= \bigotimes_{i=1}^\ell (\fv_{1}^{(i)})^{\otimes (d_i-f_i)}\). Then,
\begin{enumerate}
\item The coefficient of \(\sR^{(1)} = \bigotimes_{i=1}^\ell (\fv_{1}^{(i)})^{\otimes d_i}\) in $\sT$ is
\[
\lambda_1=\frac{1}{\fa_v\T \mathbf{C}\,\fa_u},
\qquad
\fa_u:=\mathbf{U}\T\Vect(\sR_{(f)}),\quad
\fa_v:=\mathbf{V}\T\Vect(\sR_{(d-f)}).
\]
\item Let \(\mathbf{O}_u,\mathbf{O}_v\in\RR^{r\times(r-1)}\) have orthonormal columns spanning
\(\Span\{\mathbf{C}\fa_u\}^\perp\) and \(\Span\{\mathbf{C}\fa_v\}^\perp\), respectively. Define
\[
 \hat{\mathbf{U}}:=\mathbf{U}\mathbf{O}_v,
\qquad \hat {\mathbf{C}}:=\mathbf{O}_u\T \mathbf{C} \mathbf{O}_v,\qquad \hat {\mathbf{V}}:=\mathbf{V}\mathbf{O}_u.
\]
Then
\begin{enumerate}
\item[(a)] \( \hat{\mathbf{U}},\hat {\mathbf{V}}\) have orthonormal columns and \(\hat {\mathbf{C}}\) is nonsingular;
\item[(b)] \(\hat {\mathbf{U}}\hat {\mathbf{C}}^{-1}\hat {\mathbf{V}}\T
= \mathbf{T}^{(f)}-\lambda_1\,\Vect(\sR_{(f)})\otimes\Vect(\sR_{(d-f)})\);
\item[(c)] The columns of \(\hat {\mathbf{U}}\) form an orthonormal basis of
$\{\Vect\Bigl(\bigotimes_{j=1}^\ell (\fv_{i}^{(j)})^{\otimes f_j}\Bigr) : i \geq 2 \}$

\end{enumerate}
\end{enumerate}
\end{proposition}

\subsection{Full Algorithm and Practical Considerations}
\label{subsec:practical_considerations}

The MSPM algorithm for CP decomposition of partially symmetric tensors combines subspace extraction, power iteration, and deflation. We flatten the input tensor along a well-chosen tuple $f$, extract the column span of the flattening, and iteratively recover rank-one components via PS-HOPM (Algorithm~\ref{alg:psspm}). After identifying each component, we find its complementary factors (Algorithm \ref{alg: recover complement}). We remove the recovered term using a deflation step. 
The full algorithm is Algorithm~\ref{alg:psspm-deflation}. In the remainder of this section, we cover some practical and computational aspects of the algorithm.

\begin{algorithm}[htbp]
	\caption{Multi-Subspace Power Method (MSPM)}
	\label{alg:psspm-deflation}
	\begin{algorithmic}
		\Require Tensor $\sT \in \bigotimes_{j=1}^\ell S^{d_j}(\RR^{m_j})$ with rank $r \leq r_{\max}(f)$ (or $r \leq r_{\max}(f,f')$)  and flattening tuple $f = (f_1,\dots,f_\ell)$ and $f'$ (if $f$ is not symmetry-breaking). Shift parameters $\gamma_1,\ldots,\gamma_\ell \geq0$, tolerances~$\texttt{f}_{\mathrm{tol}},\, \texttt{grad}_{\mathrm{tol}} > 0$.
		\Ensure Decomposition $\sT = \sum_{i=1}^r \lambda_i \bigotimes_{j=1}^\ell (\fv_i^{(j)})^{\otimes d_j}$
		
		\Section{Subspace Extraction}{\linewidth}
		\State Compute flattening $\mathbf{T}^{(f)} \gets \text{flatten}(\sT, f)$
		\State $(\mathbf{U}, \mathbf{D}, \mathbf{V}) \gets \texttt{svd}(\mathbf{T}^{(f)})$
		\State Update $(\mathbf{U}, \mathbf{D}, \mathbf{V})$ by truncating to top $r$ components
		\EndSection
		
		\For{$i = 1$ to $r$}
		\Section{PS-HOPM}{\linewidth}
		\State $\sT^{(f)} \gets$ reshape columns of $\mathbf{U}$ into tensors and stack 
		\State Use Algorithm~\ref{alg:psspm} to find pSVT $(\fv_i^{(1)},\ldots,\fv_i^{(\ell)})$ of $\sT^{(f)}$ with singular value one
		\State $\sR_{(i)}^{(f)} \gets \bigotimes_{j=1}^\ell (\fv_i^{(j)})^{\otimes f_j}$
		\EndSection
		
		\Section{Summand Completion}{\linewidth}
		\If{$f_i>0$ for all $i=1,\ldots,\ell$}
		\State Immediately obtain $\sR_{(i)}^{(d-f)} = \bigotimes_{j=1}^\ell (\fv_i^{(j)})^{\otimes (d_j-f_j)}$
		\ElsIf{$f$ is symmetry breaking}
		\State Use partially known factors and row span of $\mathbf{T}^{(f)}$ to obtain $\sR_{(i)}^{(d-f)}$ (see Fig.~\ref{fig:symmetrybreaking})
		\Else
		\State Use flattenings $f'$, $d - f$ to compute $\sR_{(i)}^{(d-f)}$ (see Fig.~\ref{fig:nonsymmetry})
		\EndIf
		\EndSection
		
		\Section{Deflation}{\linewidth}
		\State $\fa_u \gets \mathbf{U}\T \Vect(\sR_{(i)}^{(f)}),\quad \fa_v \gets \mathbf{V}\T \Vect(\sR_{(i)}^{(d-f)})$
		\State $\lambda_i \gets \left(\fa_v\T \mathbf{C} \fa_u\right)^{-1}$
		\If{$i < r$}
		\State $\mathbf{O}_u \gets$ orthonormal basis of $\Span\{\mathbf{C} \alpha_u\}^\perp$
		\State $\mathbf{O}_v \gets$ orthonormal basis of $\Span\{\mathbf{C} \alpha_v\}^\perp$
		\State Update $(\mathbf{U}, \mathbf{C}, \mathbf{V}) \gets (\mathbf{U} \mathbf{O}_v,\, \mathbf{O}_u\T \mathbf{C} \mathbf{O}_v,\, \mathbf{V} \mathbf{O}_u)$
		\EndIf%
		\EndSection%
		\EndFor%
		\Return $\{(\lambda_i, \fv_i^{(1)}, \ldots, \fv_i^{(\ell)})\}_{i=1}^r$
	\end{algorithmic}
\end{algorithm}

\subsubsection{Computational and storage costs}

The computational costs of Algorithm~\ref{alg:psspm-deflation} are as follows. Let $M_1 = \prod_{j=1}^\ell m_j^{a_j}$ and $M_2=\prod_{j=1}^{\ell} m_j^{d_j-a_j}$ be the number of rows and columns of $\mathbf{T}^{(f)}$, respectively. 
The computational cost in \textsc{Subspace Extraction} is dominated by $\operatorname{svd}(\mathbf{T}^{(f)})$, which is $\mathcal{O}(M_1 M_2 \min(M_1, M_2))$. However, if an upper bound for the rank $\tilde r\ge r$ is known a priori, this can be reduced to $\mathcal{O}(M_1M_2 \tilde r)$ via randomized linear algebra or iterative techniques. If the flattening $f$ is not symmetry breaking, the algorithm needs to calculate the SVD of the other flattening, with an additional cost of the same magnitude.%

The computational cost in \textsc{PS-HOPM} depends on the number of iterations until convergence. At iteration $i$ of the for loop, the successive deflation procedures have produced a tensor with $s = r - i + 1$ orthonormal slices. As explained in Section~\ref{sec: power method}, the number of multiplications needed for each power method iteration is $2 s M_1 + o(M_1)$. The computational cost of \textsc{Summand Completion} is $o(sM_2)$,
negligible compared with the other costs.

Finally, in \textsc{Deflation}, computing $\fa_u$ and $\fa_v$ cost $\mathcal{O}(s M_1)$ and $\mathcal{O}(s M_2)$ respectively, computing $\lambda_i$ is $\mathcal{O}(s^2)$ and updating $(\mathbf{U}, \mathbf{C}, \mathbf{V})$ is $\mathcal{O}(s (M_1 + M_2 + s))$.
The storage cost is also $\mathcal{O}(s(M_1 + M_2 + s))$, dominated by the storage of the SVD factorization of $\mathbf{T}^{(f)}$. If the flattening $f$ is not symmetry breaking, the SVD factorization of $\mathbf{T}^{(f')}$ also needs to be stored, incurring an additional storage cost.

\subsubsection{Low-rank approximation}

Although our theoretical analysis focuses on low-rank partially symmetric tensors, in experiments we observe empirically that MSPM succeeds in decomposing tensors that are only approximately low-rank (see Section~\ref{sec:experiments}). 

For approximately low-rank tensors, a flattening is also approximately low-rank. Its rank can be selected by examining the decay of matrix singular values and applying rank selection rules commonly used in principal component analysis \cite[Section 6.1]{Jolliffe2002PCA}. This yields a criterion for selecting the CP rank of a tensor. 

The \textsc{PS-HOPM} and \textsc{Summand Completion} routines rely on calculating pSVTs with singular value $1$. The presence of noise leads to the associated singular values having value less than $1$. Nevertheless, we observe empirically that the pSVTs obtained are still close to $1$, and the singular vectors close to the pSVTs of the true low-rank tensor. We accept pSVTs obtained by \textsc{PS-HOPM} if the singular value is $\ge 1-\zeta$, where $\zeta\ge 0$ is a hyperparameter. The stopping criterion for \textsc{PS-HOPM} is 
$$\max_{i} \|\fw^{(i)}-\fv^{(i)}\| \le \gamma,\quad\text{where}\quad  (\fw^{(1)},\dots,\fw^{(\ell)})= \text{\textsc{PMI}}(\sT; \fv^{(1)},\dots,\fv^{(\ell)}),$$
where $\gamma$ is another hyperparameter.
Finally, the approximation considerations for \textsc{Deflation} are analogous to those of SPM \cite[Section 3.3.5]{kileel2025subspace}.

\section{Numerical experiments}\label{sec:experiments}

In this section, we 
compare MSPM (Algorithm~\ref{alg:psspm-deflation}) to state-of-the-art approaches to see that it is accurate and fast. 
All experiments are run on an Apple M2 Pro with 16 GB memory.

\subsection{(2,1) Symmetry}

We generate the ground-truth tensor as 
$$\sT_{\mathrm{true}} = \sum_{i=1}^{r}\lambda_i \fa_i^{\otimes 2}\otimes \fb_i \in S^2(\mathbb{R}^{100})\otimes \RR^{50}, \quad r= 80.$$ 
The vectors $\fb_i$ are sampled with i.i.d. Gaussian entries and normalized to unit length. 
Tensors of this format appear in temporal networks~\cite{gauvin2014detecting}.
We consider two settings for $\fa_i$:  
\begin{itemize}
  \item Orthogonal case: $\fa_i$ are taken as the orthogonalized columns of a random Gaussian matrix in $\RR^{100\times 80}$, i.e. we sample $\fa_i$ via the Harr measure of the Stiefel manifold.
  \item Generic case: $\fa_i$ are sampled with i.i.d.\ Gaussian entries and normalized.  
\end{itemize}
The weights are chosen as $\lambda_i = \exp(2u_i-1)$ with $u_i \sim \mathcal{U}[0,1]$. 
We then add Gaussian noise
\[
\sT = \sT_{\mathrm{true}} + \eta \cdot \mathbscr{Z}, 
\]
scaled so that $\|\eta \cdot \mathbscr{Z}\| = 0.01 \|\sT_{\mathrm{true}}\|$, where $\mathbscr{Z}$ has i.i.d. standard Gaussian entries. 

\subsubsection{Orthogonal case}

We compare MSPM with nonlinear optimization algorithms (nonlinear least squares (NLS), unconstrained nonlinear optimization (MINF), alternating least squares (ALS)) implemented in tensorlab \cite{tensorlab3.0}, simultaneous diagonalization methods (FFDIAG \cite{ziehe2004fast}, Jacobi \cite{cardoso1996jacobi}, Jennrich), and the generalized SVD approach HOGSVD \cite{ponnapalli2011higher}.  
Due to the orthogonality, we also test SVD applied to $\sT.\mathrm{reshape}(100,5000)$, both directly and as initialization for NLS. 

We compare the recovered $\fa_i'$ with the true $\fa_i$, 
since diagonalization methods only recover the $\fa_i$.
We match the recovered $\fa_i'$ and the true $\fa_i$ via a greedy approach: we match $\fa_1$ with the vector among the $\fa_i'$s with the largest absolute cosine similarity, flip its sign if necessary, then proceed to $\fa_2$ with the remaining columns, and so on until all are matched. 
The \emph{Ascore} 
is the mean cosine similarity between the matched pairs
$$
\text{Ascore}(\{\fa_i\}_{i=1}^r, \{\fa_i'\}_{i=1}^r) = \frac{1}{r}\sum_{i=1}^r |\langle \fa_i, \fa_i'\rangle|.
$$

We plot log-runtime against log(1-Ascore) across the 10 algorithms in Figure \ref{fig:21orthogonal} using 40 randomly generated noisy tensors. 
Our MSPM method achieves both high accuracy and fast runtime. It appears in dark blue in the bottom left of the plot.

\begin{figure}[htbp]
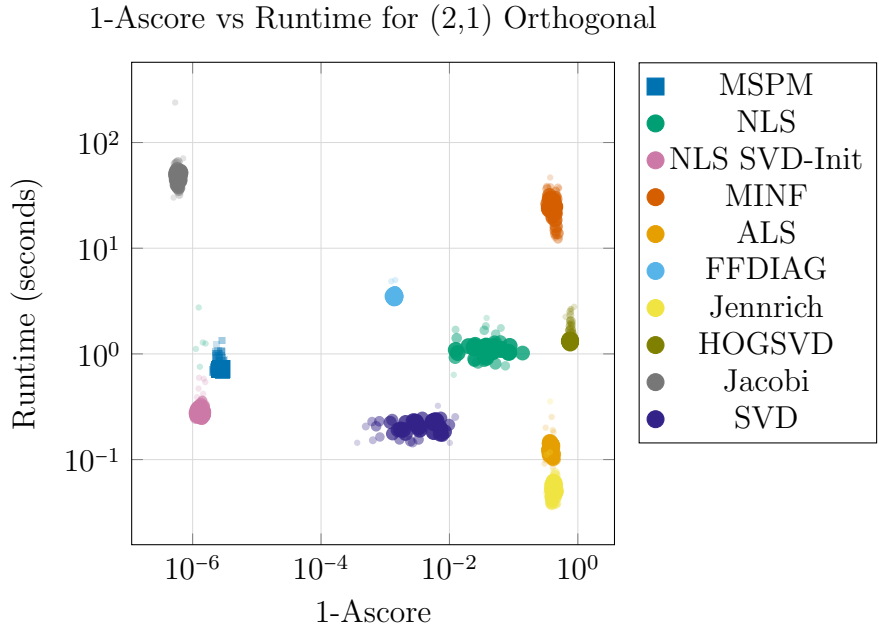

    \centering
    \includestandalone{tikzfigures_for_main/compare_21_orthogonal_fix_seed}
    \caption{Comparison of methods to decompose a $(2,1)$-partially symmetric tensor in terms of accuracy and runtime. The ground truth tensor is $\sT_{\mathrm{true}} = \sum_{i=1}^{80}\lambda_i \fa_i^{\otimes 2}\otimes \fb_i \in S^2(\mathbb{R}^{100})\times \RR^{50}$ where $\fa_i$ are orthogonal.  
    Larger, more opaque markers indicate runs that occurred with higher frequency across 40 trials. 
    }
    \label{fig:21orthogonal}
\end{figure}

\subsubsection{Generic Case}

When $\fa_i$ are drawn from normalized Gaussian vectors without orthogonalization, 
the problem is more challenging.  
We compare MSPM with nonlinear optimization algorithms (NLS, MINF, ALS) in tensorlab, 
simultaneous diagonalization methods (FFDIAG, QRJ1D \cite{afsari2006simple}, Jennrich), 
and the generalized SVD approach HOGSVD.

Accuracy is measured by Ascore, as above. 
Figure~\ref{fig:21nonorthogonal} reports log-runtime versus log(1-Ascore) across methods for 40 randomly generated noisy tensors.

\begin{figure}[htbp]
    \centering
    \includestandalone{tikzfigures_for_main/compare_21_nonortho_ellipse}
    \caption{Comparison of methods for decomposing a $(2,1)$-partially symmetric tensor in terms of accuracy and runtime. 
    The ground truth tensor is $\sT_{\mathrm{true}} = \sum_{i=1}^{80}\lambda_i \fa_i^{\otimes 2}\otimes \fb_i \in S^2(\mathbb{R}^{100})\times \RR^{50}$ where $\fa_i$ are generic Gaussian vectors normalized to unit length. The ellipses indicate the interquartile ranges (first and third quartiles) of runtime and Ascore. MSPM is the fastest among those achieving high accuracy, aside from a subset of NLS runs.  While one cluster of NLS runs attains higher accuracy than MSPM, another exhibits lower accuracy. This is reflected in the large interquartile range of NLS.
    }
    \label{fig:21nonorthogonal}
\end{figure}

\subsection{(4,1) Symmetry}

We generate the ground-truth low-rank tensor as
\[
\sT_{\mathrm{true}} = \sum_{i=1}^{r}\lambda_i \fa_i^{\otimes 4}\otimes \fb_i 
   \in S^4(\mathbb{R}^{25}) \otimes \RR^{10}, \quad r = 50,
\]
where $\fa_i$ and $\fb_i$ are sampled with i.i.d.\ Gaussian entries and normalized to unit norm.
These format of tensors appear in computational algebraic geometry, data analysis and machine learning~\cite{johnston2023computing,wang2024contrastive,usevich2025identifiability}. 
The weights are chosen as $\lambda_i = \exp(2u_i - 1)$ for $u_i \sim \mathcal{U}[0,1]$.  
We add Gaussian noise
\[
\sT = \sT_{\mathrm{true}} + \eta \cdot \mathbscr{Z},
\]
scaled so that $\|\eta \cdot \mathbscr{Z}\| = 0.01\|\sT_{\mathrm{true}}\|$, 
where $\mathbscr{Z}$ has i.i.d.\ standard Gaussian entries subject to $(4,1)$ symmetry.
For PS-HOPM, we use the $(2,1)$ flattening.  

We compare MSPM with optimization-based algorithms NLS, MINF, and ALS. In addition, we apply MSPM, NLS and MINF ignoring the partial symmetry, then restore it by averaging the learned factors. We refer to these variants as asym-avg.
We evaluate the accuracy of each algorithm by the reconstruction error
$$
\|\sT- \sum_{i=1}^{80} \lambda_i' \fa_i'^{\otimes 2} \otimes \fb_i' \|,
$$
where $\fa_i',\fb_i',\lambda_i'$ are the recovered factors and their coefficients.

Figure~\ref{fig:41} plots the runtime and reconstruction error in log-log scale for all seven methods over $40$ noisy synthetic tensors.  
Across all trials in the (4,1) case, MSPM is the most accurate and the fastest. The asym-avg variants reach similar accuracy but require longer runtimes. 
Enforcing symmetry within MSPM provides a 1.3× speedup.

\begin{figure}[htbp]
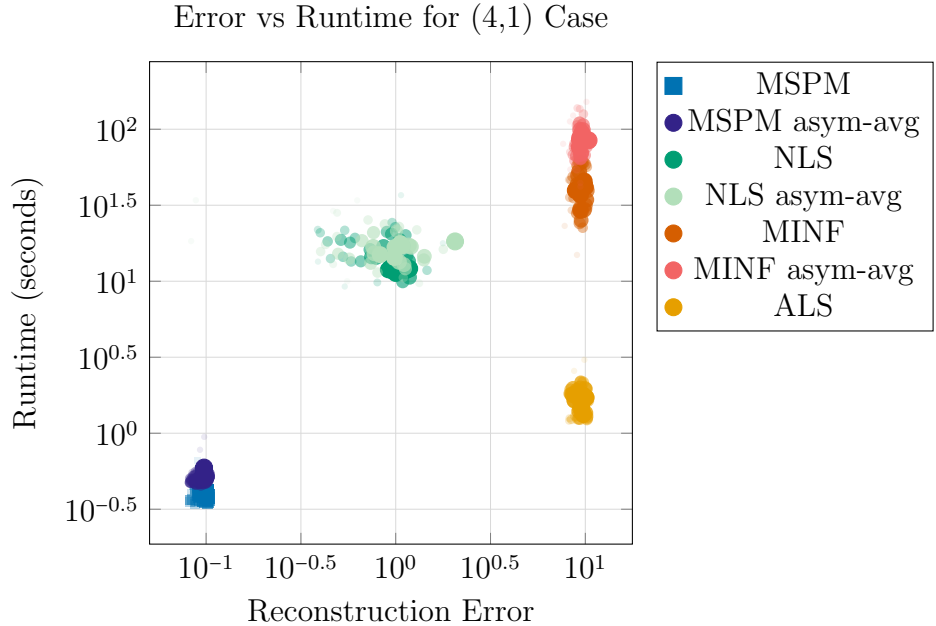

    \centering
\includestandalone{tikzfigures_for_main/compare_41_fix_seed}
    \caption{
Reconstruction error versus runtime for seven decomposition algorithms on $(4,1)$ partially symmetric tensors.  
Each point is one run on a synthetic noisy tensor of fixed size and rank.  
}
    \label{fig:41}
\end{figure}

\subsection{Order three CP decomposition}

Three-way tensors arise in a range of scientific and engineering domains~\cite{harshman1970foundations,kolda2009tensor,nickel2011three}.
We generate a ground-truth tensor $\sT_{\mathrm{true}} \in (\mathbb{R}^{100})^{\otimes 3}$ with rank $r = 80$:
\[
\sT_{\mathrm{true}} = \sum_{i=1}^{r} \lambda_i\, \fa_i \otimes \fb_i \otimes \fc_i, \quad  r=80
\]
where the $\fa_i, \fb_i, \fc_i$ are independently sampled from standard Gaussian distributions and normalized to unit norm. 
The weights are sampled as $\lambda_i = \exp(2u_i - 1)$ with $u_i \sim \mathcal{U}[0,1]$.
We add Gaussian noise
\[
\sT = \sT_{\mathrm{true}} + \eta \cdot \mathbscr{Z},
\]
scaled so that $\|\eta \cdot \mathbscr{Z}\| = 0.01\|\sT_{\mathrm{true}}\|$, 
where $\mathbscr{Z}$ has i.i.d.\ standard Gaussian entries.

We compare MSPM with six alternatives: NLS, MINF, ALS, Simulteneous diagonalization (SD) \cite{de2006link}, Simultaneous generalized Schur decomposition (SGSD) \cite{de2004computation} and Jennrich's Algorithm. 
Figure~\ref{fig:111} shows the runtime and reconstruction error (in log–log scale) of the seven algorithms over 40 synthetic noisy tensors generated as above.

\begin{figure}[htbp]
    \centering
    \definecolor{OIblue}{RGB}{0,114,178}
\definecolor{OIorange}{RGB}{230,159,0}
\definecolor{OIcyan}{RGB}{86,180,233}
\definecolor{OIgrey}{RGB}{119,119,119}
\definecolor{OIyellow}{RGB}{240,228,66}
\definecolor{OIred}{RGB}{213,94,0}

\definecolor{OIgreen}{RGB}{0,158,115}
\definecolor{OIred}{RGB}{213,94,0}

\begin{tikzpicture}
\begin{axis}[
    xlabel={Reconstruction Error},
    xtick={-1,-0.5,0,0.5,1},
    xticklabels={$10^{-1}$,$10^{-0.5}$,$10^{0}$,$10^{0.5}$,$10^{1}$},
    ylabel={Runtime (seconds)},
    ytick={-0.5,0,0.5,1,1.5},
    yticklabels={$10^{-0.5}$,$10^{0}$,$10^{0.5}$,$10^{1}$,$10^{1.5}$},
    title={Error vs. Runtime for (1,1,1) Case},
    legend style={at={(1.05,1)}, anchor=north west, 
    legend image post style={mark=*, mark size=3pt, mark options={fill opacity=1, draw opacity=1}}},
    grid=major,
    grid style={line width=.1pt, draw=gray!30},
    width=8cm,
    height=8cm,
]

\addplot[OIblue, only marks, mark=square*, mark size = 2pt, mark options = {fill = OIblue, fill opacity=0.5, draw opacity = 0}] coordinates {
    (-0.833, 0.007)     (-0.835, -0.047)     (-0.787, -0.217)     (-0.837, -0.244)     (-0.882, -0.168)
    (-0.793, -0.121)     (-0.832, -0.210)     (-0.826, -0.133)     (-0.839, -0.183)     (-0.833, -0.122)
    (-0.845, -0.133)     (-0.854, -0.076)     (-0.855, -0.234)     (-0.861, -0.145)     (-0.794, 0.084)
    (-0.817, -0.084)     (-0.855, -0.164)     (-0.887, -0.113)     (-0.865, -0.142)     (-0.839, -0.124)
    (-0.878, -0.169)     (-0.868, -0.151)     (-0.813, -0.122)     (-0.869, -0.091)     (-0.811, -0.080)
    (-0.833, -0.087)     (-0.869, -0.138)     (-0.843, -0.115)     (-0.882, 0.226)     (-0.851, -0.003)
    (-0.863, -0.110)     (-0.873, -0.089)     (-0.880, -0.093)     (-0.837, -0.027)     (-0.823, -0.005)
    (-0.821, -0.088)     (-0.825, -0.005)     (-0.859, -0.087)     (-0.827, 0.175)     (-0.844, -0.043)
};
\addlegendentry{MSPM}

\addplot[OIorange, only marks, mark=*, mark size = 2pt, mark options = {fill = OIorange, fill opacity=0.5, draw opacity = 0}] coordinates {
    (-0.017, 0.778)     (0.137, 0.834)     (0.122, 0.812)     (-0.043, 0.762)     (0.027, 0.783)
    (0.193, 0.771)     (0.019, 0.758)     (0.066, 0.763)     (0.019, 0.892)     (0.153, 0.934)
    (-0.046, 0.881)     (-0.229, 0.288)     (0.073, 0.771)     (0.097, 0.829)     (0.186, 0.840)
    (0.008, 0.793)     (0.076, 0.912)     (-0.073, 0.218)     (0.033, 0.879)     (-0.019, -0.065)
    (0.122, 0.910)     (-0.045, 0.808)     (0.134, 0.957)     (-0.219, 0.702)     (0.051, 0.887)
    (-0.057, -0.111)     (-0.117, 0.901)     (0.120, 0.894)     (0.124, 0.906)     (-0.008, 0.899)
    (-0.122, 0.135)     (-0.036, 0.687)     (0.013, 0.919)     (0.103, 0.886)     (-0.008, 0.902)
    (0.024, 0.917)     (-0.198, 0.633)     (0.046, 0.893)     (-0.080, 0.863)     (0.075, 0.654)
};
\addlegendentry{ALS}

\addplot[OIcyan, only marks, mark=*, mark size = 2pt, mark options = {fill = OIcyan, fill opacity=0.5, draw opacity = 0}] coordinates {
    (-0.859, 1.146)     (-0.846, 1.134)     (-0.821, 1.162)     (-0.851, 1.138)     (-0.899, 1.152)
    (-0.818, 1.140)     (-0.853, 1.164)     (-0.850, 1.153)     (-0.860, 1.152)     (-0.875, 1.169)
    (-0.865, 1.150)     (-0.882, 1.139)     (-0.870, 1.143)     (-0.866, 1.156)     (-0.837, 1.146)
    (-0.855, 1.147)     (-0.873, 1.143)     (-0.915, 1.153)     (-0.890, 1.143)     (-0.868, 1.149)
    (-0.900, 1.154)     (-0.896, 1.159)     (-0.847, 1.150)     (-0.875, 1.152)     (-0.848, 1.150)
    (-0.865, 1.141)     (-0.898, 1.159)     (-0.866, 1.144)     (-0.911, 1.166)     (-0.880, 1.152)
    (-0.892, 1.175)     (-0.885, 1.146)     (-0.895, 1.148)     (-0.863, 1.148)     (-0.860, 1.165)
    (-0.862, 1.154)     (-0.857, 1.152)     (-0.879, 1.156)     (-0.868, 1.162)     (-0.866, 1.146)
};
\addlegendentry{SD}

\addplot[OIgrey, only marks, mark=*, mark size = 2pt, mark options = {fill = OIgrey, fill opacity=0.5, draw opacity = 0}] coordinates {
    (-0.884, 0.505)     (-0.885, 0.516)     (-0.849, 0.501)     (-0.879, 0.503)     (-0.935, 0.508)
    (-0.847, 0.554)     (-0.880, 0.517)     (-0.876, 0.593)     (-0.893, 0.517)     (-0.900, 0.522)
    (-0.891, 0.562)     (-0.907, 0.519)     (-0.902, 0.470)     (-0.901, 0.525)     (-0.864, 0.521)
    (-0.879, 0.529)     (-0.898, 0.529)     (-0.938, 0.565)     (-0.916, 0.611)     (-0.896, 0.515)
    (-0.927, 0.560)     (-0.919, 0.562)     (-0.873, 0.510)     (-0.909, 0.521)     (-0.869, 0.516)
    (-0.891, 0.520)     (-0.920, 0.485)     (-0.886, 0.520)     (-0.937, 0.519)     (-0.908, 0.532)
    (-0.916, 0.524)     (-0.922, 0.530)     (-0.929, 0.516)     (-0.883, 0.514)     (-0.890, 0.547)
    (-0.881, 0.525)     (-0.878, 0.570)     (-0.905, 0.529)     (-0.902, 0.532)     (-0.905, 0.512)
};
\addlegendentry{SGSD}

\addplot[OIyellow, only marks, mark=*, mark size = 2pt, mark options = {fill = OIyellow, fill opacity=0.5, draw opacity = 0}] coordinates {
    (0.635, -0.509)     (0.782, -0.506)     (0.781, -0.503)     (0.695, -0.427)     (0.615, -0.508)
    (0.843, -0.508)     (0.755, -0.508)     (0.693, -0.504)     (0.736, -0.453)     (0.712, -0.438)
    (0.651, -0.495)     (0.684, -0.476)     (0.681, -0.497)     (0.891, -0.452)     (0.803, -0.478)
    (0.731, -0.482)     (0.722, -0.512)     (0.634, -0.460)     (0.661, -0.432)     (0.708, -0.432)
    (0.599, -0.504)     (0.666, -0.483)     (0.658, -0.467)     (0.585, -0.441)     (0.743, -0.457)
    (0.504, -0.464)     (0.596, -0.460)     (0.774, -0.482)     (0.612, -0.455)     (0.613, -0.427)
    (0.591, -0.296)     (0.676, -0.446)     (0.719, -0.476)     (0.636, -0.493)     (0.654, -0.457)
    (0.687, -0.433)     (0.639, -0.464)     (0.748, -0.462)     (0.721, -0.463)     (0.674, -0.466)
};
\addlegendentry{Jennrich}

\addplot[OIgreen, only marks, mark=*, mark size = 2pt, mark options = {fill = OIgreen, fill opacity=0.5, draw opacity = 0}] coordinates {
    (-0.368, 0.377)     (-0.371, 0.390)     (-0.856, 0.302)     (-0.397, 0.331)     (-0.943, 0.220)
    (-0.853, 0.317)     (-0.886, 0.370)     (-0.351, 0.370)     (-0.403, 0.296)     (-0.395, 0.411)
    (-0.388, 0.336)     (-0.398, 0.263)     (-0.909, 0.328)     (-0.908, 0.318)     (-0.870, 0.336)
    (-0.106, 0.388)     (-0.905, 0.277)     (-0.944, 0.314)     (-0.377, 0.268)     (-0.209, 0.290)
    (-0.385, 0.276)     (-0.926, 0.289)     (-0.362, 0.334)     (-0.916, 0.383)     (-0.877, 0.219)
    (-0.348, 0.401)     (-0.927, 0.319)     (-0.337, 0.287)     (-0.398, 0.226)     (-0.915, 0.483)
    (-0.923, 0.575)     (-0.928, 0.410)     (-0.937, 0.419)     (-0.890, 0.318)     (-0.243, 0.253)
    (-0.888, 0.343)     (-0.885, 0.443)     (-0.380, 0.428)     (-0.909, 0.268)     (-0.913, 0.270)
};
\addlegendentry{NLS}

\addplot[OIred, only marks, mark=*, mark size = 2pt, mark options = {fill = OIred, fill opacity=0.5, draw opacity = 0}] coordinates {
    (0.975, 1.255)     (1.003, 1.215)     (0.982, 1.148)     (0.777, 1.658)     (0.962, 1.111)
    (0.959, 1.385)     (0.880, 1.623)     (0.907, 1.508)     (0.982, 1.307)     (1.010, 1.321)
    (0.951, 1.225)     (0.998, 1.220)     (0.984, 1.145)     (1.095, 0.342)     (1.019, 1.346)
    (0.980, 1.249)     (1.021, 1.248)     (0.985, 1.320)     (0.964, 1.013)     (0.991, 1.410)
    (0.908, 1.598)     (0.966, 1.408)     (1.123, 0.572)     (0.970, 1.281)     (1.032, 1.404)
    (0.807, 1.656)     (0.854, 1.656)     (1.110, 0.593)     (0.937, 1.421)     (0.968, 1.306)
    (0.921, 1.442)     (0.961, 1.351)     (0.959, 1.428)     (0.779, 1.608)     (1.011, 1.376)
    (1.049, 1.132)     (0.976, 1.674)     (0.944, 1.348)     (0.978, 1.345)     (1.091, 0.349)
};
\addlegendentry{MINF}
\draw[fill, opacity=0.2, draw=none, OIblue] (-0.845,-0.106) ellipse (0.018 and 0.037);
\addplot[only marks, mark=square*, OIblue, mark size=1.5pt] coordinates { (-0.8425,-0.08) };

\draw[fill, opacity=0.2, draw=none, OIorange] (0.027,0.828) ellipse (0.071 and 0.067);
\addplot[only marks, mark=*, OIorange, mark size=1.5pt] coordinates { (0.029,0.7956) };

\draw[fill, opacity=0.2, draw=none, OIcyan] (-0.870,1.151) ellipse (0.013 and 0.005);
\addplot[only marks, mark=*, OIcyan, mark size=1.5pt] coordinates { (-0.869,1.151) };

\draw[fill, opacity=0.2, draw=none, OIgrey] (-0.896,0.524) ellipse (0.015 and 0.008);
\addplot[only marks, mark=*, OIgrey, mark size=1.5pt] coordinates { (-0.8952,0.528) };

\draw[fill, opacity=0.2, draw=none, OIyellow] (0.684,-0.474) ellipse (0.048 and 0.021);
\addplot[only marks, mark=*, OIyellow, mark size=1.5pt] coordinates { (0.69429,-0.467) };

\draw[fill, opacity=0.2, draw=none, OIgreen] (-0.645,0.334) ellipse (0.265 and 0.050);

\draw[fill, opacity=0.2, draw=none, OIred] (0.977,1.321) ellipse (0.028 and 0.102);
\addplot[only marks, mark=*, OIred, mark size=1.5pt] coordinates { (0.975,1.358) };

\end{axis}
\end{tikzpicture}
    \caption{Comparison of decomposition methods for asymmetric tensors in $\RR^{100 \times 100 \times 100}$ with rank 80. 
    Each point is a single run.
    Ellipses indicate the interquartile ranges (first and third quartiles). MSPM exhibits stable performance and the fastest runtime among the accurate methods.
    } 
    \label{fig:111}
\end{figure}

To assess scalability, we consider a tensor of size $500\times 500\times 500$ with rank $r=400$, comparing MSPM, SGSD, SD, and NLS. The SD method does not terminate due to memory requirements. The runtimes and reconstruction errors are reported in Table~\ref{tab:large_scale}. MSPM achieves the fastest runtime while maintaining competitive accuracy.

\begin{table}[htbp]
\centering
\begin{tabular}{lcc}
\toprule
Method & Runtime (s) & Reconstruction Error \\
\midrule
MSPM & 163.6  & 0.2826 \\
SGSD & 1938.2 & 0.2795 \\
NLS  & 485.7  &  1.1332 \\
SD   & ---    & --- \\
\bottomrule
\end{tabular}
\caption{Performance on a $500\times 500\times 500$ tensor with rank $400$. 
SD does not terminate due to memory limitations.}
\label{tab:large_scale}
\end{table}

\section{Conclusion}\label{sec:conclusion}

We introduced the \emph{Multi-Subspace Power Method (MSPM)}, an algorithm for decomposing    low-rank tensors with arbitrary partial symmetries.  
It transforms a general input tensor into one with orthonormal slices. 
We showed that this transformation enables a large class of tensors to be decomposed via their partially symmetric singular vector tuples (pSVTs) of singular value one.
Further, we developed a \emph{Partially Symmetric Higher-Order Power Method (PS-HOPM)} to compute pSVTs and proved its global convergence to first-order critical points as well as local linear convergence in the two-block case to second-order critical points. 
We demonstrated that MSPM is a competitive  algorithm compared to prior methods, through numerical experiments on noisy tensors with symmetry types that are common in applications.  

Several open questions remain. One direction is to determine whether the local linear convergence result in Proposition~\ref{prop:local linear} can be extended to partially symmetric tensors with more than two blocks. 
Another question is whether the relationship between pSVTs and rank-one summands extends to tensors of higher rank, potentially revealing deeper algebraic or geometric structure in partially symmetric tensor decompositions.

\vspace{1.75em}

\begin{footnotesize} 
\noindent \textbf{Acknowledgments.}  
A.S. thanks Yue M. Lu for helpful discussions.
A.S. and K.W. thank Stephanie Hsu for implementing MSPM algorithms in PyTorch.

\noindent \textbf{Funding Declarations.}
A.S. was supported in part by the Sloan Foundation. J.P. was supported in part by a start-up grant from the University of Georgia. J.K. was supported in part by NSF DMS 2309782, NSF DMS 2436499, NSF CISE-IIS 2312746,
DE SC0025312, and the Sloan Foundation.

\noindent \textbf{Declarations}
The authors declare no conflicts of interest and no competing interests.
\end{footnotesize}

\bibliographystyle{plain}
\bibliography{reference}

\appendix


\section{The choice of flattening}
\label{sec:flattening}

\subsection{Proof of Proposition \ref{thm:rmax_formula}}\label{app: proof of rmax}

In this section, we prove Proposition \ref{thm:rmax_formula}.

Recall that we form a flattening matrix $\mathbf{T}^{(f)}$ from $\sT$, and then construct a tensor $\sT^{(f)}$ by reshaping and stacking an orthonormal basis of its column space. The pSVTs of $\sT^{(f)}$ with singular value one are in one-to-one correspondence with rank-one tensors in this column space, provided that $\rank(\sT^{(f)}) = \rank(\mathbf{T}^{(f)})$.

Each column of $\mathbf{T}^{(f)}$ is (the vectorization of) a partially symmetric tensor in $\bigotimes_{i=1}^\ell S^{f_i}(\mathbb{R}^{m_i})$. 
The set of rank-one tensors (up to scale) in this space is the Segre–Veronese variety of $\PP^{m_1-1}\times \cdots \times \PP^{m_\ell-1}$ embedded via multi-degree \( \mathcal{O}(f_1, \ldots, f_\ell) \), which we denote by $X \subseteq \mathbb{P}^{N-1}$, where $N = \prod_{i=1}^{\ell} \binom{m_i + f_i - 1}{f_i} = n_{\rm row}^f.$
The dimension of \( X \) is~\( \sum_{f_i\neq 0} (m_i - 1) \) and its codimension is
\begin{equation}
    n_{\rm codim}^f := \prod_{i=1}^{\ell} \binom{m_i + f_i - 1}{f_i} - 1 - \sum_{f_i \neq 0} (m_i-1).
\end{equation}
Since $X \subset \PP^{n_{\rm row}^f-1}$ and $m_i\geq 2$, we have $n_{\rm codim}^f < n_{\rm row}^f$. 
The roles of the rows and columns in $\mathbf{T}^{(f)}$ are not on equal footing: 
we seek rank-one tensors in the column span.

We will see that $\sT^{(f)}$ has a unique rank-$r$ decomposition when 
$\colspan\mathbf{T}^{(f)}$ contains exactly $r$ rank-one tensors (up to scale) with the prescribed partial symmetry.
This holds whenever $r \leq n_{\rm codim}^f$.  
It may hold with positive Lebesgue measure when 
$r = n_{\rm codim}^f + 1$, see \cite{ranestad2024real}.
The above discussion of $r(f)$ is formalized in Proposition \ref{thm:rmax_formula} with the following proof.

\begin{proof}[Proof of Proposition~\ref{thm:rmax_formula}]
Consider a rank $r$ tensor 
$\sT \in \bigotimes_{i=1}^\ell S^{d_i}(\RR^{m_i})$ with decomposition 
\[
\sT
= \sum_{i=1}^r \lambda_i \bigotimes_{j=1}^\ell \bigl(\fv^{(j)}_i\bigr)^{\otimes d_j},
\qquad
\lambda_i \in \RR\setminus\{0\},\ \fv^{(j)}_i \in \mathbb{S}^{m_j-1},
\]
where the coefficients $\lambda_i$ and the factors $\{\fv^{(j)}_i\}$ are generic.
For a fixed tuple $f=(f_1,\ldots,f_\ell)$, the flattening $\mathbf{T}^{(f)}$ is
\[
\mathbf{T}^{(f)}
= \sum_{i=1}^r \lambda_i \;
\underbrace{\Vect \Bigl(\bigotimes_{j=1}^\ell (\fv^{(j)}_i)^{\otimes f_j}\Bigr)}_{=\,\mathbf{a}_i}
\ \otimes\
\underbrace{\Vect \Bigl(\bigotimes_{j=1}^\ell (\fv^{(j)}_i)^{\otimes (d_j-f_j)}\Bigr)}_{=\,\mathbf{b}_i}.
\]
If $r \leq\min\{n_{\rm row}^f,n_{\rm col}^f\}$, then the vectors $\{\mathbf{a}_i\}_{i=1}^r$ and $\{\mathbf{b}_i\}_{i=1}^r$ are linearly independent, since this is a Zariski open condition on the factors $\fv^{(j)}_i$ and it is nonempty because the Segre-Veronese varieties are nondegenerate.  
Hence $\operatorname{rank}(\mathbf{T}^{(f)}) = r$ and
\[
\operatorname{colspan}(\mathbf{T}^{(f)})
= \Span\{\mathbf{a}_1,\ldots,\mathbf{a}_r\}
= \Span \Bigl\{\Vect\Bigl(\bigotimes_{j=1}^\ell (\fv^{(j)}_i)^{\otimes f_j}\Bigr): i=1,\ldots,r\Bigr\}.
\]

If $r\leq r(f)$, we show that the tensor $\sT^{(f)}$ has a unique rank $r$ decomposition, since $\bigotimes_{j=1}^\ell (\fv^{(j)}_i)^{\otimes f_j}$ for $i=1,\ldots,r$ are the only rank-one tensors in $\colspan \bT^{(f)}$ up to scale. 
Let $X\subseteq \PP^{N-1}$ denote the the Segre–Veronese variety parametrizing rank-one tensors in $\bigotimes_{j=1}^\ell S^{f_j}(\RR^{m_j})$ up to scale. 
If
\[
\dim \operatorname{colspan}(\mathbf{T}^{(f)}) + \dim X < N - 1
= \prod_{i=1}^{\ell} \binom{m_i + f_i - 1}{f_i} - 1,
\]
which is equivalent to $r \leq n_{\rm codim}^f$, then the intersection 
\(
\operatorname{colspan}(\mathbf{T}^{(f)}) \cap X
\)
consists of the $r$ points
\(
\bigotimes_{j=1}^\ell(\fv^{(j)}_i)^{\otimes f_j}
\)
for $i=1,\ldots,r$,
by the Generalized Trisecant Lemma~\cite[Proposition 2.6]{chiantini2002weakly}, 
When $f$ has two entries $i,j$ equal to one, with $\min\{m_i,m_j\}=2$, 
and all other entries zero, we have $X \cong \PP^{m_i-1} \times \PP^{m_j-1}$ with degree $\deg X = \max\{m_i,m_j\}$. 
When $r = n_{\rm codim}^f +1 = \max\{m_i,m_j\}$, the linear space $\operatorname{colspan}(\mathbf{\bT}^{(f)})$ has complementary dimension to $X$, so by Bézout’s theorem~\cite[Theorem 18.3]{harris2013algebraic} the intersection consists of exactly $\deg X=r$ points. For all other tuples $f$, we have $\deg X>r$ and there exist linear space of projective dimension $r-1$ that intersects $X$ in more than $r$ real points, see~\cite[Theorem 1.6]{ranestad2024real}.

If $r>r(f)$, there is an Euclidean open set of linear spaces $\operatorname{colspan}(\mathbf{\bT}^{(f)})$ that intersect $X$ in
more than $r$ real rank-one tensors \cite[Theorem 1.6]{ranestad2024real}. 
We can choose a different basis of $\operatorname{colspan}(\mathbf{\bT}^{(f)})$ consisting of rank-one tensors. This yields another rank $r$ decomposition of $\sT^{(f)}$, contradicting uniqueness.
\end{proof}

\subsection{Maximal admissible rank for given tensor type}
In this section, we describe how to choose the flattening to maximize the admissible rank of $\sT$.

The admissible rank of $\sT$ under MSPM depends on the flattening tuple $f$, and, in the non-symmetry-breaking case, also on auxiliary tuples such as $f'$ and $d-f$ used in the completion step, see Corollary~\ref{cor: r_max with matching}.
We seek flattenings that maximize the admissible rank.
\begin{definition}\label{def:r}
A tuple $f$ or a pair of tuples $f,f'$ (when $f$ is not symmetry-breaking) is called \emph{optimal} if it attains the maximal recovery rank, i.e.,
\[
r_{\max}(f) \quad \text{or} \quad r_{\max}(f,f')
\]
is maximal over all flattenings. 
\end{definition}

We obtain formulae for the optimal flattening for $(\ell,1)$–symmetric tensors with $\ell=2,3,4$, which arise in contrastive data analysis \cite{abid2018exploring,wang2024contrastive, wang2026multi}.

\begin{corollary}
Consider a partially symmetric tensor in 
$S^\ell(\mathbb{R}^n) \otimes \mathbb{R}^k$.  
The optimal flattening tuple depends on $(n,k)$ as follows. For \( \ell = 2\) and \( 3 \), the optimal tuple is \( f = (1,1) \).
     For \( \ell = 4 \), the optimal tuple is
    \[
    f = \begin{cases}
        (2,1), & k \leq \dfrac{n^2+3n-2}{2n-2}, \\[0.8em]
        (1,1), & k > \dfrac{n^2+3n-2}{2n-2}.
    \end{cases}
    \]
\end{corollary}

\begin{proof}
We restrict to flattenings $(k,1)$, since all other tuples need to pair with some $(k,1)$ tuple, which yields smaller $r_{\max}$, by Corollary~\ref{cor: r_max with matching}.
For $\ell=2$, the only tuple with format $(*,1)$ is $(1,1)$, which is completable. Thus, $(1,1)$ is optimal.
For $\ell=3$, we have
\begin{itemize}
\item 
   $ r(1,1) = 
    \begin{cases}
    \min\{\binom{n+1}{2}, (n-1)(k-1)\}, & \min\{n,k\}>2, \\[0.3em]
    \min\{3,k\}, & n=2, \\[0.3em]
    n, & k=2.
    \end{cases}$
    \item \( r(2,1) = \min \left\{ n, \binom{n+1}{2}k -n-k+1 \right\}=n \).
\end{itemize}
Since $r(1,1)\ge r(2,1)$ in all cases, the tuple $(1,1)$ is optimal.

For $\ell=4$, the three candidates are $(1,1), (2,1)$ and $(3,1)$. They yield
\begin{itemize}
\item \(
r(1,1) =
\begin{cases}
\min\!\left\{ \binom{n+2}{3},\ (n - 1)(k - 1) \right\}, & n > 2, \\[0.6ex]
\min\{4,\,k\}, & n = 2, \\[0.6ex]
\min\{4,\,n\}, & k = 2.
\end{cases}
\)
\item \( r(2,1) = \min \left\{ \binom{n+1}{2},\ \binom{n+1}{2}k + 1 - n - k \right\} = \binom{n+1}{2} \)
\item \( r(3,1) = \min \left\{ n,\ \binom{n+2}{3}k + 1 - n - k \right\} = n \)
\end{itemize}
We have $r(3,1)\leq r(2,1)$, so $(3,1)$ cannot be optimal.
When $n=2$ or $k=2$, $(2,1)$ dominates. 
For $n>2,k>2$, $(1,1)$ becomes optimal once $(n-1)(k-1)>\binom{n+1}{2}$, i.e.
\[
k > \frac{n^2+3n-2}{2n-2}.
\qedhere \]
\end{proof}

\section{Computing pSVTs}
Our partially symmetric tensor decomposition computes pSVTs with singular value one of a tensor $\sT^{(f)}$ 
with last slices orthonormal.
In computations, to speed up convergence, we introduce a tensor $\widetilde{\sT}^{(f)}$ such that the pSVTs of $\sT^{(f)}$ and $\widetilde{\sT}^{(f)}$ are in one-to-one correspondence. 

The tensor $\widetilde{\sT}^{(f)}$ is a symmetrized combination of the orthonormal slices of $\sT^{(f)}$. 
The shift parameters for $\widetilde{\sT}^{(f)}$ can be chosen to be less than one, which is smaller than the shifts in Theorem \ref{thm:power_method_convergence}.
Smaller shifts imply a faster increase in the function value and hence faster convergence.
For numerical computations, it is not necessary to form $\widetilde{\sT}^{(f)}$ explicitly, since we only require its contractions with rank-one tensors, which can be computed using $\sT^{(f)}$.

\begin{definition}[Partial symmetrization $P_{\mathrm{sym}}$ and tensor $\widetilde{\sT}^{(f)}$]\label{def: Ttilde}
Fix a partially symmetric tensor 
$\sT^{(f)} \in \bigotimes_{i=1}^\ell S^{f_i}(\RR^{m_i}) \otimes \RR^r$
whose last slices $\sS_1,\dots,\sS_r$ are orthonormal.  
Define the \emph{partial symmetrization operator} $P_{\mathrm{sym}}$ to be the linear projection
\[
P_{\mathrm{sym}} \colon\,
\Bigl( \bigotimes_{i=1}^\ell S^{f_i}(\RR^{m_i}) \Bigr)^{\otimes 2}
\;\to\;
\bigotimes_{i=1}^\ell S^{2f_i}(\RR^{m_i}).
\]
Define \(
\widetilde{\sT}^{(f)} := \sum_{i=1}^r P_{\mathrm{sym}}(\sS_i \otimes \sS_i)\).
\end{definition}

Notice the order of $\widetilde{\sT}^{(f)}$ is double the order of $\sT^{(f)}$. The next proposition gives shifts $\gamma_i$ 
that ensure global convergence of Algorithm~\ref{alg:psspm} applied to $\widetilde{\sT}^{(f)}$.

\begin{proposition}\label{prop: shift for Ttilde A}
Fix a partially symmetric tensor $\sT^{(f)} \in \bigotimes_{i=1}^\ell S^{f_i}(\RR^{m_i}) \otimes \RR^r$ 
whose last slices $\sS_1, \ldots, \sS_r$ are orthonormal. Let $\widetilde{\sT}^{(f)}$ be defined as in Definition \ref{def: Ttilde}.
Then Algorithm~\ref{alg:psspm} applied to $\widetilde{\sT}^{(f)}$ converges globally to a pSVT of $\widetilde{\sT}^{(f)}$ with shifts
\begin{enumerate}
    \item[(a)] $\gamma_i = 0$ when $f_i = 1$.
    \item[(b)] $\gamma_i = \sqrt{\tfrac{f_i-1}{f_i}}\,h(\nu)$ when $f_i > 1$, 
    where $F(\fv^{(1)},\ldots,\fv^{(\ell)}) \leq\nu$ and
    \[
    h(\nu) =
    \begin{cases}
    1 - \tfrac{\nu}{2}, & \nu \leq\tfrac{2}{3}, \\[0.5em]
    \sqrt{2\nu(1-\nu)}, & \nu > \tfrac{2}{3}.
    \end{cases}
    \]
\end{enumerate}
\end{proposition}

\begin{proof}
We show that the shifts make 
\begin{equation}\label{eq:F-tildeA}
F(\fv^{(1)},\ldots,\fv^{(\ell)})
:= \Big\langle \widetilde{\sT}^{(f)},\;\bigotimes_{i=1}^\ell (\fv^{(i)})^{\otimes 2f_i} \Big\rangle
\;+\;\sum_{i=1}^\ell \gamma_i \,\|\fv^{(i)}\|^{2f_i}
\end{equation}
convex in each block. Global convergence of Algorithm~\ref{alg:psspm} then follows from Theorem~\ref{thm:power_method_convergence}.
Consider block $i$.
If $f_i=1$, then
the $i$-th block of $F$
is a quadratic form hence convex, so $\gamma_i=0$ suffices. This proves (a).
When $f_i>1$, Lemma~\ref{lem: T tilda A} gives the identity
$$
F_\Acal(\fv^{(1)},\ldots,\fv^{(\ell)})=\Big\langle \widetilde{\sT}^{(f)}, \;\bigotimes_{i=1}^\ell (\fv^{(i)})^{\otimes 2f_i}\Big\rangle = \sum_{i=1}^r \langle \sS_i,\bigotimes_{i=1}^\ell (\fv^{(i)})^{\otimes f_i}\rangle^2.
$$
Its Hessian is
$$
\nabla_i^2F_\Acal = 2f_i^2\sum_{i=1}^r \bigl(\sS_i \cdot \bigotimes_{j=1}^\ell (\fv^{(j)})^{\otimes (f_i-\delta_{ij})}\bigr)^{\otimes 2}
+ 2f_i(f_i-1)\sum_{i=1}^r\langle \sS_i,\bigotimes_{i=1}^\ell (\fv^{(i)})^{\otimes f_i}\rangle \sS_i\cdot\bigotimes_{j=1}^\ell (\fv^{(j)})^{\otimes (f_i-2\delta_{ij})}.
$$
It follows that 
$$
\nabla^2_i F = \nabla_i^2F_\Acal + \nabla^2_i \sum_{i=1}^\ell \gamma_i \|\fv^{(i)}\|^{2f_i} = \nabla_i^2F_\Acal + 4f_i(f_i-1)\gamma_i\|\fv^{(i)}\|^{2f_i-4}(\fv^{(i)})^{\otimes 2} + 2f_i \gamma_i \|\fv^{(i)}\|^{2f_i-2}I.
$$
Let $\mathbscr{S}_{j\setminus i}(\fv)$ denote $\sS_j\cdot \bigotimes_{k\neq i} (\fv^{(k)})^{\otimes f_k} $.
For any $\fy \in \SB^{m_i-1}$, we obtain
\begin{align*}
\frac{1}{2f_i}\, \fy\T \nabla^2_i F(\fv^{(1)},\ldots,\fv^{(\ell)})\,\fy
 &= f_i\sum_{j=1}^r \Big\langle \mathbscr{S}_{j\setminus i}(\fv),\ (\fv^{(i)})^{\otimes f_i-1}\otimes \fy \Big\rangle^2 \\
&\quad + (f_i-1) \sum_{j=1}^r \langle \sS_i,\bigotimes_{i=1}^\ell (\fv^{(i)})^{\otimes f_i}\rangle \langle \mathbscr{S}_{j\setminus i}(\fv),\ (\fv^{(i)})^{\otimes f_i-2}\otimes \fy^{\otimes 2}\rangle 
\\
&\quad +2(f_i-1)\gamma_i \langle \fv^{(i)},\fy\rangle^2+ \gamma_i\\
& \geq (f_i-1) \sum_{j=1}^r \langle \sS_i,\bigotimes_{i=1}^\ell (\fv^{(i)})^{\otimes f_i}\rangle \langle \mathbscr{S}_{j\setminus i}(\fv),\ (\fv^{(i)})^{\otimes f_i-2}\otimes \fy^{\otimes 2}\rangle + \gamma_i\\
& = (f_i-1) \langle P_\Acal(\bigotimes_{i=1}^\ell (\fv^{(i)})^{\otimes f_i})\cdot \bigotimes_{j\neq i} (\fv^{(j)})^{\otimes f_j}, (\fv^{(i)})^{\otimes f_i-2}\otimes \fy^{\otimes 2}\rangle + \gamma_i.
\end{align*}

We note that
$$
\langle P_\Acal(\bigotimes_{i=1}^\ell (\fv^{(i)})^{\otimes f_i})\cdot \bigotimes_{j\neq i} (\fv^{(j)})^{\otimes f_j}, (\fv^{(i)})^{\otimes f_i} \rangle 
= \|P_\Acal(\bigotimes_{i=1}^\ell (\fv^{(i)})^{\otimes f_i})\|^2.
$$
We write the vector $\fy$ as 
\(\fy = \alpha \fx + \beta \bar{\fy}\), where \(\bar{\fy}\perp \fx\), and use the same argument as for \cite[Lemma 4.4]{kileel2025subspace} to obtain
\[
\frac{1}{2f_i}\, \fy\T \nabla^2_i F(\fv)\,\fy
\;\ge\; -\sqrt{\tfrac{f_i-1}{f_i}}\,h(\nu) + \gamma_i,
\]
where $h(\nu)$ is as defined in the statement.  
\end{proof}

The shifts for $\widetilde{\sT}^{(f)}$ are strictly less than 1 and can be smaller than the shifts in Theorem \ref{thm:power_method_convergence}.
To connect the construction of $\widetilde{\sT}^{(f)}$ with our optimization problem, we introduce a projection operator $P_\Acal$ and objective function $F_\Acal$, which measures how close a rank-one tensor is to the subspace $\Acal$.
We then relate $F_\Acal$ to the tensors $\sT^{(f)}$ and $\widetilde{\sT}^{(f)}$.

\begin{definition}[Projection and objective function]\label{def:functions}
Let $\sT^{(f)}\in \bigotimes_{i=1}^\ell S^{f_i}(\RR^{m_i})\bigotimes \RR^r$ be a partially symmetric tensor with last slices orthonormal that span $\Acal \subseteq \bigotimes_{i=1}^{\ell} S^{f_i}(\RR^{m_i})$.

\begin{itemize}
    \item The projection $P_\Acal$ is the orthogonal projection onto $\Acal$ of a tensor or its vectorization.
    \item The objective function $F_\Acal$ quantifies proximity to $\Acal$:
    \[
    F_\Acal \colon \,\,\prod_{i=1}^\ell \mathbb{S}^{m_i-1} \to \mathbb{R}, 
    \quad 
    F_\Acal(\fv^{(1)},\ldots,\fv^{(\ell)}) 
= \Bigl\| P_\Acal\,\Bigl( \bigotimes_{i=1}^\ell \bigl(\fv^{(i)}\bigr)^{\otimes f_i} \Bigr) \Bigr\|^2.
    \]
\end{itemize}
\end{definition}

We relate pSVTs to the objective $F_\Acal$. The pSVTs of $\sT^{(f)}$ are in one-to-one correspondence with those of $\widetilde{\sT}^{(f)}$, which are precisely the global maximizers of $F_\Acal$.

\begin{lemma}\label{lem: T tilda A}
Fix a tensor $\sT^{(f)} \in \bigotimes_{i=1}^\ell S^{f_i}(\RR^{m_i}) \otimes \RR^r$ 
whose last slices $\sS_1, \ldots, \sS_r$ are orthonormal with span $\Acal$.
\begin{enumerate}
    
    \item[(1)] The tuple $(\fv^{(1)},\ldots,\fv^{(\ell)}) \in \mathbb{S}^{m_1-1}\times\cdots\times \mathbb{S}^{m_\ell-1}$ is a pSVT of $\widetilde{\sT}^{(f)}$ with singular value $\sigma^2$ if and only if $(\fv^{(1)},\ldots,\fv^{(\ell)},\fw)$ is a pSVT of $\sT^{(f)}$ with singular value $\sigma$, where
    \(
    \fw \) is the vector \( \sT^{(f)} \cdot \bigotimes_{i=1}^\ell {(\fv^{(i)})}^{\otimes f_i} \), rescaled to be norm one.

    \item[(2)] 
  The  critical points of $F_\Acal$ are the pSVTs of $\widetilde{\sT}^{(f)}$.
    The function $F_\Acal$ takes values in $[0,1]$ with $1$ attained when $\bigotimes_{i=1}^\ell (\fv^{(i)}){}^{\otimes f_i}\in \Acal$.

\end{enumerate}
\end{lemma}

\begin{proof}
(1)
We establish the correspondence between the pSVTs of $\widetilde{\sT}^{(f)}$ and  $\sT^{(f)}$.  
Fix $k \in \{1,\ldots,\ell\}$.  
For a tuple $(\fv^{(1)},\ldots,\fv^{(\ell)})$, we compute
\begin{align}
\widetilde{\sT}^{(f)} \cdot \bigotimes_{i=1}^\ell (\fv^{(i)})^{\otimes (2f_i - \delta_{i,k})}
&= \sum_{j=1}^r \left\langle \sS_j, \bigotimes_{i=1}^\ell (\fv^{(i)})^{\otimes f_i}\right\rangle
\cdot \left(\sS_j \cdot \bigotimes_{i=1}^\ell (\fv^{(i)})^{\otimes (f_i - \delta_{i,k})}\right)\\
& = 
\left\|\sT^{(f)}\cdot \bigotimes_{i=1}^\ell (\fv^{(i)})^{\otimes f_i}\right\| 
\sum_{j=1}^r \fw_j \cdot \left(\sS_j \cdot \bigotimes_{i=1}^\ell (\fv^{(i)})^{\otimes (f_i - \delta_{i,k})}\right) \\
\label{eq:Ta Ttildea}
&=\left\|\sT^{(f)}\cdot \bigotimes_{i=1}^\ell (\fv^{(i)})^{\otimes f_i}\right\| \sT^{(f)} \cdot \bigotimes_{i=1}^\ell (\fv^{(i)})^{\otimes (f_i - \delta_{i,k})}\otimes \fw.
\end{align} 
If $(\fv^{(1)},\ldots,\fv^{(\ell)})$ is a pSVT of $\widetilde{\sT}^{(f)}$ with singular value $\sigma^2$, then
\(
\widetilde{\sT}^{(f)} \cdot \bigotimes_{i=1}^\ell (\fv^{(i)})^{\otimes (2f_i - \delta_{i,k})} 
= \sigma^2 \fv^{(k)}\).
Equivalently, 
\[
\sT^{(f)} \cdot \bigotimes_{i=1}^\ell (\fv^{(i)})^{\otimes (f_i - \delta_{i,k})} \otimes \fw
= \sigma \fv^{(k)},
\]
since
\(
\Bigl\| \sT^{(f)} \cdot \bigotimes_{i=1}^\ell (\fv^{(i)})^{\otimes f_i} \Bigr\|
= \Bigl( \widetilde{\sT}^{(f)} \cdot \bigotimes_{i=1}^\ell (\fv^{(i)})^{\otimes 2f_i}\Bigr)^{1/2}
= \sigma\).
The definition of $\fw$ also says that
\(
\sT^{(f)} \cdot \bigotimes_{i=1}^\ell (\fv^{(i)})^{\otimes f_i} = \sigma \fw.
\)
Hence $(\fv^{(1)},\ldots,\fv^{(\ell)},\fw)$ satisfies the pSVT equations for $\sT^{(f)}$ with singular value $\sigma$.  

(2) The contraction \( \widetilde{\sT}^{(f)} \cdot \bigotimes_{i=1}^\ell (\fv^{(i)})^{\otimes 2f_i} \) can be rewritten as
\begin{align*}
\sum_{j=1}^r \left\langle P_{\mathrm{sym}}(\sS_j \otimes \sS_j), \bigotimes_{i=1}^\ell (\fv^{(i)})^{\otimes 2f_i}\right\rangle 
&= \sum_{j=1}^r \left\langle \sS_j \otimes \sS_j, P_{\mathrm{sym}}\!\left(\bigotimes_{i=1}^\ell (\fv^{(i)})^{\otimes 2f_i}\right)\right\rangle \\
&= \sum_{j=1}^r \left\langle \sS_j, \bigotimes_{i=1}^\ell (\fv^{(i)})^{\otimes f_i}\right\rangle^2 \\
&= \left\| P_\Acal \left(\bigotimes_{i=1}^\ell (\fv^{(i)})^{\otimes f_i}\right) \right\|^2 = F_\Acal(\fv^{(1)},\ldots,\fv^{(\ell)}).
\end{align*}
The critical points of the function
\[
(\fv^{(1)},\ldots,\fv^{(\ell)}) \;\mapsto\; 
\widetilde{\sT}^{(f)} \cdot \bigotimes_{i=1}^\ell (\fv^{(i)})^{\otimes 2f_i}
= F_\Acal(\fv^{(1)},\ldots,\fv^{(\ell)})
\]
are the pSVTs of $\widetilde{\sT}^{(f)}$, by \cite{lim2005singular}.
We have $F_\Acal(\fv^{(1)},\ldots,\fv^{(\ell)}) \in [0,1]$, 
by definition of $F_\Acal$ and $\|\bigotimes_{j=1}^\ell (\fv^{(j)})^{\otimes f_j}\|=1$. 
It follows from (1) that the global maxima of $F_\Acal$ are the pSVTs of $\sT^{(f)}$ with singular value one. 
The tensor $\bigotimes_{i=1}^\ell (\fv^{(i)}){}^{\otimes f_i}$ lies in $\Acal$, by Theorem \ref{thm:decomp and singular vector}.
\end{proof}

\end{document}